\documentclass{amsart}

\usepackage{amsmath,amsfonts,amssymb,amsthm}
\usepackage[numbers]{natbib}
\usepackage{bm}
\let\cal\mathcal

\usepackage{bibunits}
\usepackage{anyfontsize}
\usepackage[colorlinks=true,
            linkcolor=blue,
            citecolor=blue,
            urlcolor=blue]{hyperref}
\usepackage[margin=1.5in]{geometry}

\DeclareMathOperator*{\argmin}{arg\,min}

\def\N{\mathbb{N}}

\def\R{\mathbb{R}}
 
\def\E{\mathbb{E}}

\usepackage{relsize}
\usepackage{tikz}

\theoremstyle{plain}

\newtheorem{theorem}{Theorem}[section]
\newtheorem{lemma}[theorem]{Lemma}
\newtheorem{remark}[theorem]{Remark}

\newtheorem{Problem}{Problem}

\begin{document}

\title[Distributional Limit Theory for OT]{Distributional Limit Theory for Optimal Transport}
\thanks{The research of Eustasio del Barrio is partially supported by PID2021-128314NB-I00 and PID2024-162240NB-I00, funded by MCIN/AEI/10.13039/501100011033/FEDER}

\author{Eustasio del Barrio}
\address{Eustasio del Barrio is Professor, IMUVa,
Universidad de Valladolid, Spain.}
\email{eustasio.delbarrio@uva.es}

\author{Alberto González-Sanz}
\address{Alberto González-Sanz is Assistant Professor, Department of Statistics, Columbia University, New York, USA.}
\email{ag4855@columbia.edu}

\author{Jean-Michel Loubes}
\address{Jean-Michel Loubes is Research Director at INRIA, Toulouse, France.}
\email{loubes@math.univ-toulouse.fr}

\author{David Rodríguez-Vítores}
\address{David Rodríguez-Vítores is post-doc researcher at IMUVa, , Universidad de Valladolid, Spain.}
\email{david.rodriguez.vitores@uva.es}

\begin{abstract}
 Optimal Transport (OT) is a resource allocation problem with applications in biology, data science, economics and statistics, among others. In some of the applications, practitioners have access to samples which approximate the continuous measure. Hence the quantities  of interest derived from OT --- plans, maps and costs --- are only available in their empirical versions.  Statistical inference on OT aims at finding confidence intervals of the population plans, maps and costs,  which, in recent years, have gained an increasing interest in the statistical community. In this paper we provide a comprehensive review of the most influential results on this research field. 
\end{abstract}


\maketitle
\markboth{DEL BARRIO, GONZÁLEZ-SANZ, LOUBES, AND RODRÍGUEZ-VÍTORES}
         {DISTRIBUTIONAL LIMIT THEORY FOR OPTIMAL
TRANSPORT}
\section{Introduction}
Optimal transport (OT) is, by now, a regular member of the standard toolkit in many fields of Data Science. The list includes computer vision \cite{BonneelDigne2023}, flow cytometry gating \cite{del2020optimalflow,freulon2023cytopt}, domain adaptation \cite{chang2022unified}, fair learning \cite{deLara.et.al.JMLR.2024,gordaliza2019obtaining} or generative modeling in AI \cite{pooladian2024pluginestimationschrodingerbridges}, to name only a few examples.
While optimal transport (OT) has long been a fundamental problem in mathematics, its recent surge in popularity within data science is largely driven by computational advances, particularly the introduction of the Sinkhorn algorithm \cite{Cuturi2013SinkhornDL, Peyre.Cuturi.2019.Book}, which has significantly improved the scalability of OT-based methods{.}
The use of optimal transport (OT) for inferential purposes has a long history, but until recently, 
it was \textit{`hindered by the lack of distributional limits'} \cite{SommerfeldMunk2018}. 
While there had been some early attempts to apply OT in goodness-of-fit problems \cite{delBarrioetal1999a, delBarrioGineUtzet2005, freitag2005hadamard,Munk.Czado.1998.SeriesB}, these efforts were largely restricted to the univariate setting, limiting their broader applicability.
To be more precise, 
let us fix some notation. We limit our presentation to the Euclidean setup. We recall that the OT problem is the minimization problem
\begin{align}\label{eq:plans}
\mathcal{T}_c(P,Q):=\min_{\pi \in \Pi(P,Q)}\int_{\R^d\times \R^d} c(x,y) d \pi(  {x},   {y}),
\end{align}
where $P$ and $Q$ are Borel probabilities on $\R^d$, $c:\R^d\times\R^d\to \R$ is a  cost function and $\Pi(P,Q)$ stands for the set of joint probabilities on $\R^d\times\R^d$ with marginals $P$ and $Q$. We will write $c_p(x,y)=\|x-y\|^p$ and $\mathcal{T}_p$ instead of $\mathcal{T}_c$ when $c=c_p$, $p\geq 1$. An {\it OT plan} is a minimizer in \eqref{eq:plans}. Under mild assumptions (see \cite[Theorem~5.10]{villani2008optimal}), the optimization problem \eqref{eq:plans} admits a dual formulation given by
\begin{align}\label{eq:potentials}
\mathcal{T}_c(P,Q)=\max_{f,g}\int f({x})dP(x)+\int g(  {y}) dQ(  {y}),
\end{align}
where the supremum is taken over the class $\Phi_c(P,Q)$ consisting of pairs $(f,g)\in L_1(P)\times L_1(Q)$ such that  
$$ f(  {x})+g(  {y})\leq c(  {x},  {y}),\quad  x,y\in\mathbb{R}^d.$$ A pair of {\it OT potentials} $(f_0,g_0)$ is any maximizer  in  \eqref{eq:potentials}. Under some assumptions on the cost function and the probability measures it is well known (see, e.g., Theorem 2.44 in \cite{villani}) that the optimal plan in \eqref{eq:plans} is of the form $\pi=({\rm I}_d,T)\sharp P$, where $T$ is such that
\begin{eqnarray}\label{eq:maps}
\mathcal{T}_c(P,Q)&=&\int_{\R^d} c(x,T(x)) dP(x)\\
\nonumber
&=&\min_{S:\, S\sharp P=Q}\int_{\R^d} c(x,S(x)) dP(x).
\end{eqnarray}
Such a minimizer is called an \textit{OT map}. Here, and throughout this paper, $T\sharp P$ denotes the push-forward of $P$ by $T$, namely, the probability  {defined by $(T\sharp P)(A)=P(T^{-1}(A))$ for { every} Borel set $A$.}

The basic objects of interest in the goodness-of-fit problems cited above were the empirical versions of $\mathcal{T}_c(P,Q)$, namely, given the observation of an i.i.d.\ sample  $X_1, \dots, X_n\sim P$ we look at $\mathcal{T}_c(P_n,Q)$, where $P_n$ denotes the empirical measure on the sample (for the sake of brevity we focus here on one-sample problems, but all the results that we present can be adapted to the two-sample case). This paper will present an essentially complete description of the distributional limit theory for $\mathcal{T}_c(P_n,Q)$. To avoid excessive technicalities we limit ourselves to the case of potential costs, $c_p$, $p\geq 1$, which has received a great deal of attention in the literature. The related \textit{Wasserstein distance}, $\mathcal{W}_p(P,Q)=(\mathcal{T}_p(P,Q))^{\frac{1}{p}}$ is a metric on the set of probability measures on $\mathbb{R}^d$ with finite moment of order $p$ that metrizes weak convergence of probabilities plus convergence of $p$-th moments. We refer again to \cite{villani,villani2008optimal} for these and other background results on OT. Distributional limit theorems for $\mathcal{T}_p(P_n,Q)$ and related functionals are the basis of the already cited applications of OT in {goodness-of-fit problems} \cite{ delBarrioetal1999a, delBarrioGineUtzet2005,freitag2007nonparametric,Munk.Czado.1998.SeriesB} { or in change point analysis \cite{10.1214/20-AOS2036}}.

We will also look at the empirical OT potentials, plans and maps. These are objects of interest by themselves. The OT map to or from a reference measure is the center-outward distribution or quantile function of \cite{Chernozhukovetal2017,Hallinetal2021}. These maps can be used to define multivariate ranks which can be used to build efficient inference procedures in a semiparametric or nonparametric setup, see  {Section~\ref{Section:quantiles}} or  \cite{MR4646617,deb2023pitmanefficiencylowerbounds,Deb-Sen.JASA} and references therein.   This gives additional motivation to look for distributional limit theorems for these OT maps and we try to give a short account about the available results on this topic.

An already well-known fact about OT is that estimation suffers from the curse of dimensionality, either if the goal is to estimate the OT map (see \cite{Hutter.Rigollet.2021.AOS}) or the cost \cite{niles2022estimation,manole.weed.2024.AoAP}. { We summarize here the main results in this regard for the squared Euclidean cost. We define 
$$ \beta(n,d)=\begin{cases}
n^{-1} & \text{if } d=1, \\
\log(n) n^{-1} & \text{if } d=2, \\
 n^{-\frac{2}{d}} & \text{if } d\geq 3. 
\end{cases}
$$ 
}
{\begin{theorem}(Curse of dimensionality) \label{theo:curse}
Assume that $P$ and $Q$ are supported in convex sets and admit H\"older continuous densities bounded away from zero and infinity in their supports. 
Then the following hold
\begin{enumerate}
    \item If $P=Q$,
    $ \E[\mathcal{T}_2(P_n,P)] \lesssim  \beta(n,d) .$
    \item If the OT map (for the squared Euclidean cost) $T$ from $Q$  to $P$ is $\mathcal{C}^1$ and strongly monotone, 
    $$ | \E[\mathcal{T}_2(P_n,Q)]- \mathcal{T}_2(P,Q)| \lesssim  \beta(n,d) $$
    and 
    $$ \E[\|T- \hat{T}\|^2_{L^2(Q)}] \lesssim  \beta(n,d), $$
    where $\hat{T}$ denotes the OT map from $Q$  to $P_n$. 
\end{enumerate}
\end{theorem}
The label `curse of dimensionality' in Theorem~\ref{theo:curse} refers to the fact that the rate $\beta(n,d)$ is sharp up to logarithmic factors \cite{Ambrosio-EJS,Hutter.Rigollet.2021.AOS,niles2022estimation}, in the sense that there exist $P,Q\in\mathcal{P}_2(\mathbb{R}^d)$ for which $| \E[\mathcal{T}_2(P_n,Q)]- \mathcal{T}_2(P,Q)| \gtrsim  \beta(n,d)$, the same result remaining true if $\mathcal{T}_2(P_n,Q)$ were replaced by an alternative estimator.
Theorem~\ref{theo:curse} is the compilation of several results. Firstly, \cite{Fournier.Guillin.2014.PTRF} showed that 
$$ \E[\mathcal{T}_2(P_n,P)] \lesssim  \begin{cases}
n^{-1/2} & \text{if } d=1,2,3, \\
\log(n) n^{-1/2} & \text{if } d=4, \\
 n^{-\frac{2}{d}} & \text{if } d\geq 5. \\
\end{cases}$$
Secondly, \cite{Ajtai-1984} showed point 1 of  Theorem~\ref{theo:curse} for $P$ being the uniform measure on the 2-dimensional unit square (see also \cite{AmbrosioStraTrevisan,Ledoux_conjecture}).  This was improved later by \cite{Benedetto.et.al.,Ambrosio-EJS} to the one displayed in Theorem~\ref{theo:curse},~part~1. The second statement follows from \cite[Theorem~6]{Manole.et.al.2024.AoS}.  Convergence rates for different costs can be found in \cite{Fournier.Guillin.2014.PTRF}.
}

The curse of dimensionality  can be alleviated through different paths. {The first one is by {\it a priori} smoothness assumptions on the OT map. }  The minimax rate of estimation of the OT map { and cost} improves under smoothness {(see \cite{Hutter.Rigollet.2021.AOS} for the map and \cite{Niles-Weed-Berthet.22.AOS,Singh.et.al.2018.Nips} for the cost)} and one can look for estimators that adapt to smoothness (and remain computationally tractable), see \cite{Manole.et.al.2024.AoS}. Some recent work provides distributional limit theorems for these improved estimators \cite{Manole.et.al.2024.Preprint}. 

Smoothness is not the only possible workaround to the curse of dimensionality in OT estimation. The discrete setup is rich enough for an interesting set of applications and has been extensively analyzed (see \cite{SamworthJohnson2004,SommerfeldMunk2018,TamelingSommerfeldMunk2019}). Additionally, OT satisfies a \textit{lower complexity adaptation} principle (see \cite{Hundrieser.et.al.2024.AIHP}). Roughly speaking, this means that the statistical complexity of OT estimation is determined by the statistical complexity in the estimation of the ``easy'' measure, $P$ or $Q$. {This enables distributional
limit theorems in low dimension} or in the semidiscrete setup (see \cite{delBarrioetAl.2024.BJ,Hundrieser.et.al.2024.AIHP}).

{Beyond the smooth and low-dimensional regimes}, alternative formulations of the OT problem have received a great deal of attention in the literature. Entropic optimal transport (see \cite{Cuturi2013SinkhornDL, Peyre.Cuturi.2019.Book}) is a very attractive choice, both in computational terms, thanks to the celebrated Sinkhorn algorithm, and in statistical complexity (see \cite{delBarrio.et.al.2023.SIMODS.Improved,Mena.Weed.2019.Nips}). The estimation of the entropic OT cost (EOT cost in the sequel) is not affected by the curse of dimensionality and this opens the possibility, for instance, to build asymptotically valid confidence intervals for the EOT cost in any dimension.

An alternative way to benefit from OT data analysis tools {is to consider OT between
projections which mitigates the issues inherent to high-dimensional problems}. Of all the possible approaches, the sliced Wasserstein metric, looking at one-dimensional projections, has received most attention in the literature (see \cite{Goldfeld.et.al.2022.statisticalinferenceregularizedoptimal,han2024max,Manole.et.al.EJS,Xi.Weed.2022.Nips,xu2022central}). In this line we comment briefly on the related CLTs, including some possible improvements over the available literature.

{In this work, we present a review of the distributional limit theorems available and a discussion on the different proof techniques leading to them. As a general rule, different arguments are needed for $p=1$ and $p>1$, as well as for the null case $P=Q$ and the non-null case $P\neq Q$. The latter distinction is also relevant for the statistical applications discussed above: null results are naturally related to goodness-of-fit testing \cite{ delBarrioetal1999a, delBarrioGineUtzet2005,freitag2007nonparametric,Munk.Czado.1998.SeriesB}, whereas non-null results are the basis for confidence intervals \cite{delBarrio.et.al.2023.SIMODS.Improved,gordaliza2019obtaining,Manole.et.al.EJS}.} While we do not aim to give a complete account of the technical details underlying the proofs of all the results that we summarize here, we find it convenient to include a few key guidelines to the main approaches. The OT problem, either in its primal \eqref{eq:plans} or dual form \eqref{eq:potentials} is a linear minimization problem over a convex set. However, from the point of view of distributional limit theory, other formulations of the problem may be more convenient. The Monge formulation \eqref{eq:maps} may look more direct, but the problem becomes highly non-linear and some linearization technique could provide a useful approach to the problem. Different possibilities have been considered. In the case of discrete probability measures the transportation cost functional is directionally Hadamard differentiable. This fact, together with a delta method for nonlinear derivatives is enough to yield CLTs for transportation cost{s} (see \cite{SommerfeldMunk2018, TamelingSommerfeldMunk2019}). Hadamard differentiability of supremum type functionals (see \cite{CarcamoCuevasRogriguez2020}) can also be applied to obtain CLTs for the empirical transportation cost in the semidiscrete setup (see \cite{delBarrioetAl.2024.BJ,Hundrieser.et.al.BJ.2024.Unifying}). Additionally, this approach can yield distributional limit theorems for OT {potentials.} On the other hand, if we put the main focus on the transportation cost, other linearization tools can give sharper results. The Efron-Stein inequality for variances of functions of independent random elements (see, e.g., \cite{MR3185193}) turns out to be a very convenient tool in this field. As an example of this we can refer to Lemma~\ref{Efron_Stein_L1} in the supplementary material to this paper: if $P$ and $Q$ are Borel probability measures on $\mathbb{R}^d$, $Q$ has finite mean and $P$ has finite variance $\sigma^2(P)$ then a trivial application of the Efron-Stein inequality shows that
\begin{equation}\label{l1_var_bound_quote}
{\rm Var}(\mathcal{T}_1(P_n,Q))\leq \frac {\sigma^2(P)}n.
\end{equation}
A slightly more elaborate application of the Efron-Stein inequality yields a similar bound for the case $p>1$ (under the assumption that $P$ and $Q$ have finite moments of order $2p$ and $Q$ has a density; see Theorem 3.1 in \cite{delBarrio.Loubes.2019.AoP.CLTquadratic} and Corollary 4.3 in \cite{delBarrio.GonzalezSanz.Loubes.AIHP.2024.CLTgeneral}).
As a consequence, we see that $\sqrt{n}(\mathcal{T}_p(P_n,Q)-\mathbb{E}[\mathcal{T}_p(P_n,Q)])$ is stochastically bounded under very mild assumptions. This suggests to look for CLTs for the fluctuation of the empirical transportation cost with respect to its expected value and, indeed, we can also use the Efron-Stein inequality as a linearization tool to obtain CLTs for this fluctuation with great generality, as we will show in later sections.

The remaining sections of this paper are organized as follows. Section~\ref{SECTION:1D} presents CLTs for one-dimensional OT objects. While this may look very narrow in scope, some of the key features of the theory show up in this simpler setup: CLTs for the fluctuation of the OT cost with respect to its expected value hold with great generality; Gaussianity of the limiting distribution of the empirical OT cost is related to the regularity of the cost; 
limit theorems for the OT plans and maps require more restrictive 
assumptions. These aspects are revisited in Section~\ref{general_dimension}  for general dimension. Here we pay special attention to those cases in which OT is not affected by the curse of dimensionality (low-dimensional, discrete or semi-discrete setup), linking the results to the lower complexity adaptation principle of OT observed in \cite{Hundrieser.et.al.2024.AIHP}. Section~\ref{Section:regularized} presents distributional limit theory for regularized OT. We consider entropic and smooth OT. From  {the} point of view of practitioners EOT is most attractive due to the availability of efficient computational tools. We present limit theorems for the cost, potentials and plans and also for the related Sinkhorn divergence. In the case of smooth OT we deal with the cost and also with some estimators of the OT map that adapt to smoothness. This is a very interesting field of research since, theoretically, one can construct estimators of the OT map that adapt to smoothness and avoid the curse of dimensionality. However, the available theory covers a somewhat limited setup and we discuss possible ways to extend it to cover more natural cases. Regrettably, this part of the discussion is very technical, but we believe that the topic is interesting enough to include this material in the paper. We continue with Section~\ref{Section:Further-OT}, devoted to sliced { and Gromov--Wasserstein }distances. These combine some convenient features of one-dimensional OT with the ability to capture differences in multivariate distributions. We review the available results, paying special attention to some recent improvements over the existing literature. 
While we focus mainly on results and outline proof techniques, we include a small section (Section~\ref{section:applications}) discussing the potential applications of some of the distributional results presented in this paper. Finally, we include a section with a small sample of open problems which we believe deserve further investigation.

\section{One-dimensional central limit theorems}\label{SECTION:1D}
We consider in this section empirical objects in OT problems on the real line. A crucial simplification here is that, if $P$ and $Q$ are probabilities on the real line with distribution functions $F$ and $G$, respectively, and $p\geq 1$, then
\begin{equation}\label{1dquantiles}
\mathcal{T}_p(P,Q)=\int_0^1|F^{-1}(t)-G^{-1}(t)|^pdt,
\end{equation}
where $F^{-1}$, $G^{-1}$ denote the corresponding quantile functions. A further simplification holds in the case $p=1$, since
\begin{equation}\label{1dempirical}
\mathcal{T}_1(P,Q)=\int_{\mathbb{R}}|F(x)-G(x)|dx,
\end{equation}
that is, $\mathcal{T}_1(P,Q)$ is simply the $L_1(\mathbb{R})$-norm of the difference between distribution functions. The link between OT and distribution and quantile functions goes further. If $P$ has no atoms then $F$ is the OT map between $P$ and the uniform distribution on $(0,1)$, $U(0,1)$. In all cases, $F^{-1}$ is the OT map between $U(0,1)$ and $P$. While the empirical d.f., say $F_n$, is not the OT map between the empirical measure, $P_n$, and $U(0,1)$ (such an object does not exist since $P_n$ is discrete), it is a consistent estimator of the OT map $F$ in different senses. $F_n^{-1}$ is the OT map between $U(0,1)$ and $P_n$. All these facts turn the study of distributional limit laws for empirical estimators of OT maps on the real line into exercises about weak convergence of empirical or quantile processes in some $L_p$ space. However, weak convergence of these processes typically requires stronger assumptions than weak convergence of the transportation cost, which, as we will see next, holds under very mild assumptions if we look at the fluctuation between the empirical cost and its expected value. Later in this section we discuss the role of the centering constants in the CLTs for the transportation cost. We complete the section with a study of CLTs for OT maps and plans in this one-dimensional setup.

\subsection{CLTs for the fluctuation}

In the one-dimensional case ($d=1$) the fluctuation of the empirical transportation cost is fully understood through the following result. {Here ${\rm sgn}(x)$ equals $-1, 0$ or $+1$ for $x<0$, $x=0$ or $x>0$, respectively.}
{
\begin{theorem}\label{1dCLT} Fix $d=1$. Assume that $P$ and $Q$ have finite moments of order $2p$ and that $Q$ has a continuous quantile function.
Then the following hold: 
\begin{enumerate}
    \item If $p>1$, \begin{equation}\label{unbiased_CLT_1d}
\sqrt{n}(\mathcal{T}_p(P_n,Q)-\mathbb{E}[\mathcal{T}_p(P_n,Q)])\underset w \to N(0,\sigma_p^2),
\end{equation}
where $\sigma_p^2=\sigma_p^2(P,Q)\geq 0$, defined in \eqref{limiting_variance_1d}, is strictly positive unless $P=Q$ or $P$ is Dirac's measure on a point. 
\item If $p=1$, \begin{equation}\label{unbiased_CLT_1d_p1}
\sqrt{n}(\mathcal{T}_1(P_n,Q)-\mathbb{E}[\mathcal{T}_1(P_n,Q)])\underset w \to \gamma(P,Q),
\end{equation}
where \begin{equation}\label{def_G_PQ}
\gamma(P,Q):=\int_{\mathbb{R}} (v_{F,G}(x)-\mathbb{E}[v_{F,G}(x)])dx,
\end{equation}
for 
\begin{equation}   \label{def_v_FG}
    v_{F,G}(x)= \mathbb{I}_{\{F(x)=G(x)\}}|B\circ F(x)|
+ {\rm sgn}(F(x)-G(x)) B\circ F(x)
\end{equation}
and $B$ is a standard Brownian bridge on $[0,1]$. Furthermore, the distribution of $\gamma(P,Q)$ is Gaussian if $\ell(F=G)=0$.
\end{enumerate}
\end{theorem}
}
The case $p>1$ in Theorem \ref{1dCLT} is Theorem 2.1 (ii) in \cite{CLTonedimensional}. { The case $p=1$ and $P=Q$ was considered in \cite{delBarrioGineMatran1999}, under the stronger assumption 
\begin{equation}\label{assumption:L21}
\int_{\R}\sqrt{F(x)(1-F(x))}dx<\infty,
\end{equation}
related to the interpolation space $L_{2,1}$: it holds if $P$ has a finite moment of order $2+\delta$ for some $\delta>0$; it implies finite moment of order 2.
Under \eqref{assumption:L21} the trajectories of the process $B\circ F$ belong a.s. to $L_1(\mathbb{R})$. In this case and assuming $P=Q$ we have that 
\begin{equation}\label{1dL21sameP}
\sqrt{n}\,\mathcal{T}_1(P_n,Q)\underset w \to \int_{\mathbb{R}}|B\circ F(x)| dx.
\end{equation}
This is part of Theorem 1.1 in \cite{delBarrioGineMatran1999}. Theorem 5.1 in the same reference covers the case $P=Q$ under the weaker assumption that $P$ has a finite variance. To our knowledge, the case $p=1$ at the level of generality included in Theorem \ref{1dCLT} is new, covering the case of $P(\ne Q)$ of finite variance, but not necessarily satisfying \eqref{assumption:L21}. A proof is given in the supplementary material, including the fact that {\eqref{def_G_PQ}} defines indeed an a.s. finite random variable.}

Theorem \ref{1dCLT}, while limited in scope, brings to the focus a few significant issues. First, the fluctuation of the empirical OT cost satisfies a CLT under very mild moments and smoothness assumptions: no need for bounded support or light tails (for the sake of brevity we refrain from pursuing the case of heavier tails, but the literature covers it in some cases, see \cite{delBarrioGineMatran1999,delBarrioGineUtzet2005}). 
We observe that in the case $p>1$, if $Q$ is Dirac's measure on a point then the CLT for $\mathcal{T}_p(P_n,Q)$ is simply the CLT for sums of i.i.d.\ random variables and a finite moment of order $2p$ is  then a necessary condition for the conclusion. The assumption about continuity of the quantile function rules out the possibility of a disconnected support: the support of $Q$ is connected if and only if $G^{-1}$ is continuous, see Proposition A.7 in \cite{BobkovLedoux2019}. The different regime of the  empirical transportation cost in this case has been observed in the literature for a long time (see \cite{SamworthJohnson2004, BobkovLedoux2019}). 

A further point of interest in Theorem \ref{1dCLT} is that non strictly convex cost functions (the case $p=1$) result in some differences with respect to the CLT. We see that non-Gaussian limiting distributions can show up. In fact, this will be always the case if $P=Q$. The limiting distribution will be the law of the centered version of the right-hand side in \eqref{1dL21sameP}. In particular, we see here a non-degenerate limiting distribution. In contrast, the conclusion of Theorem  \ref{1dCLT} for $p>1$ and $P=Q$ is simply 
\begin{equation}\label{vanishing_variance}
\sqrt{n}(\mathcal{T}_p(P_n,P)-\mathbb{E}[\mathcal{T}_p(P_n,P)])\underset {\text{\small Pr.}} \to 0.
\end{equation}

As a summary to this subsection we can say that CLTs for the empirical transportation cost on the real line, with a Gaussian limiting distribution, hold with great generality. That is, under minimal moment and smoothness 
conditions. However, these minimal assumptions come at the price of somewhat unnatural centering constants: the centering constant in \eqref{unbiased_CLT_1d} is not $\mathcal{T}_p(P,Q)$, as one would desire in some statistical applications. We deal with this issue next.

\subsection{The role of the centering constants}
Given the interest in Wasserstein metrics, one could look for, say, a confidence interval for $\mathcal{W}_p(P,Q)$ or for $\mathcal{T}_p(P,Q)$. 
If we could ensure that 
\begin{equation}\label{biased_CLT_1d}
\sqrt{n}(\mathcal{T}_p(P_n,Q)-\mathcal{T}_p(P,Q))\underset w \to N(0,\sigma_p^2(P,Q)),
\end{equation}
then, provided that $\hat{\sigma}_p^2(P,Q)$ is a consistent estimator of $\sigma_p^2(P,Q)$, we would have that
$$\Big[\mathcal{T}_p(P_n,Q)\pm \frac{\hat{\sigma}_p(P,Q)}{\sqrt{n}}\Phi^{-1}(1- {\textstyle \frac \alpha 2})\Big]$$
is an {approximate} confidence interval for $\mathcal{T}_p(P,Q)$ 
with asymptotic level $1-\alpha$. The key to  move from \eqref{unbiased_CLT_1d} to the desired version \eqref{biased_CLT_1d} would be to prove that
\begin{align}
\label{remove_bias_W2_one_sample}
    \sqrt{n} \left( \mathbb E\left(\mathcal{T}_p(P_n,Q)\right) - \mathcal{T}_p(P,Q)\right) \rightarrow 0.
\end{align} 
Theorem 2.3 in \cite{CLTonedimensional} provides sufficient conditions to replace this centering constant with $\mathcal{T}_p(P,Q)$. This result was recently improved by {Theorem 4.2} in \cite{sliced_arxiv}. We present here an adaptation of this result for the one-dimensional setup. The key assumption can be formulated in terms of the following quantities:
\begin{equation*}
J_p(P):=\int_0^1 \frac{(t(1-t))^{\frac{p}{2}}}{f^p(F^{-1}(t))}dt, \quad p\geq 1.
\end{equation*}
Here $F$ is the distribution function of $P$, which is assumed to have density $f$.  

\begin{theorem}\label{teor:bias_1d}
Assume $p>1$ and let $P,Q\in \mathcal P_p(\mathbb R)$ be such that $P$ has distribution function $F$ and density $f$ with $f(F^{-1}(t))$ positive and continuous in $(0,1)$, and monotone for $t$ sufficiently close to $0$ and $1$.  If there exist $\alpha,\beta>1 $ with $\frac{1}{\alpha}+\frac{1}{\beta}=1$ 
such that $J_{\alpha}(P) < \infty$ and  
{$Q$, if $p=2$, or $P,Q$, if $p\neq 2$, have finite moment of order $\beta\max(1,p-1)$, 
then \eqref{remove_bias_W2_one_sample} holds.}

If $p=1$ and $J_1(P)<\infty$ then 
$$\sqrt{n} \left( \mathbb E\left[\mathcal{T}_1(P_n,Q)\right] - \mathcal{T}_1(P,Q)\right) \rightarrow c,
$$
with $c\geq 0$ a finite constant. If $\ell(F=G)=0$, then $c=0$ and \eqref{remove_bias_W2_one_sample} holds.
\end{theorem} 

Details of the proof are given in the supplementary material. As a consequence of Theorem~\ref{teor:bias_1d} we see that we can indeed replace $\mathbb{E}[\mathcal{T}_p(P_n,Q)]$ with $\mathcal{T}_p(P,Q)$ in \eqref{unbiased_CLT_1d} or \eqref{unbiased_CLT_1d_p1} under some additional assumptions. We should recall at this point that finiteness of $J_p(P)$ is necessary and sufficient for boundedness of $n^{\frac{p}{2}}\mathbb{E}[\mathcal{T}_p(P_n,P)]$ and also for boundedness of 
$\sqrt{n}\mathbb{E}\left[\left(\mathcal{T}_p(P_n,P)\right)^{\frac{1}{p}} \right]$ (this is Corollary 5.10 in \cite{BobkovLedoux2019}). Hence, for $p>1$ we see that if $J_p(P)<\infty$ then
$$ \sqrt{n} \left( {\mathbb E}\left[\mathcal{T}_p(P_n,P)\right] - \mathcal{T}_p(P,P)\right) \rightarrow 0.$$
Theorem~\ref{teor:bias_1d} improves upon this known result dealing with the case $P\ne Q$.

\subsection{The null case}

As noted above, in the case $P=Q$ the limiting variance in \eqref{unbiased_CLT_1d} vanishes. Under the additional assumption $J_p(P)<\infty$ we obtain (for $p>1$) that
$$\sqrt{n}\mathcal{T}_p(P_n,P)\underset {\text{\small Pr.}} \to 0.$$
One may wonder if it is possible to get a non-degenerate limiting distribution with a different rate. Combining different approaches (see Theorem 6.4.1 in \cite{CsorgoHorvath1993} and Theorem 3.2 in \cite{SamworthJohnson2004}; { we refer also to \cite{10.1214/20-AOS2036} and the references therein for related results}) the following result follows.
\begin{theorem}\label{1dCLT_nullcase}
Assume $P$ has a continuous density $f$ such that $f\circ F^{-1}$ is continuous in $(0,1)$ and monotonic in a neighbourhood of $0$ and $1$. If for some  $p\geq 1$, \begin{equation}\label{J_p_assumption_null}
J_p(P):=\int_0^1 \frac{(t(1-t))^{\frac{p}{2}}}{f^p(F^{-1}(t))}dt<\infty.
\end{equation}
then 
\begin{equation*}\label{null_case_1d}
n^{\frac{p}{2}}\mathcal{T}_p(P_n,P)\underset w \to \int_0^1 \frac{|B(t)|^{p}}{f^p(F^{-1}(t))}dt,
\end{equation*}
where $B$ is a standard Brownian {bridge} on $[0,1]$.
\end{theorem}
Again, we defer details of the proof to the supplementary material. Finiteness of $J_p(P)$ is a necessary and sufficient condition for a.s. integrability of the process in the limiting integral. It is also related to integrability of some functionals of the hazard function. This result does not bring any improvement over \eqref{1dL21sameP} in the case $p=1$. { For $p>1$ it can be shown that \eqref{J_p_assumption_null} requires that $P$ is supported in an interval and $f$ should be a.e. positive on it (see Section 5.2 in \cite{BobkovLedoux2019}). It also 
implies finiteness of the moment generating function in a neighborhood of the origin (see Section~3 in \cite{SamworthJohnson2004}). Hence, this is a much stronger assumption than finiteness of moments of order $2p$ in Theorem~\ref{1dCLT}. In simpler terms, $J_p(P)$ is large when $P$ has nearly disconnected support or has sufficiently heavy tails.} Under some relaxations of \eqref{J_p_assumption_null} it is still possible to get a version of {Theorem \ref{1dCLT_nullcase}}, although with the addition of some centering terms. We refer to \cite{delBarrioGineUtzet2005} for details.

\subsection{Limits for optimal transportation plans and maps.}

When $p>1$ and one of the {marginals} admits a density (this can be relaxed, see Theorem 2.44 in \cite{villani}) the OT plan is given by a map, that is, is of type $\pi=({\rm I}_d,T)\sharp P$. This general result has also a particularly simple formulation in the one-dimensional case, since the joint probability $(F^{-1},G^{-1})\sharp \ell_{(0,1)}$, {where $\ell_{(0,1)}$ denotes uniform Lebesgue measure on $[0,1]$,} is an optimal plan for all the $c_p$ cost functions, $p\geq 1$. We recall that $F^{-1}$ is also the unique optimal transportation map from {$\ell_{(0,1)}$} to $P$. Hence, distributional limit theory for optimal plans and maps can, in this case, be formulated in terms of the natural estimator of the quantile function, namely, the empirical quantile function, $F_n^{-1}$. This is a left-continuous, piecewise constant map, with $F_n^{-1}(t)=X_{(i)}$ for $\frac{i-1}n<t\leq \frac i n$, $i=1,\ldots,n$, where $X_{(i)}$ denotes the $i$-th order statistic associated to the sample $X_1,\ldots,X_n$ of i.i.d.\ observations from $P$. Now, the natural object to look at is the \textit{quantile process}
$$v_n(t):=\sqrt{n}(F_n^{-1}(t)-F^{-1}(t)), \quad 0<t<1.$$
It is well known that $v_n(t)$ is asymptotically Gaussian if $P$ has a density such that $f(F^{-1}(t))>0$. The natural setup to look for a functional CLT for $v_n$ is, in turn, $L_p(0,1)$: if we prove that $v_n\underset w \to \mathbb{V}$ for some $L_p(0,1)$-valued random element then we would get, by the continuous mapping theorem that $n^{\frac{p}{2}}\mathcal{T}_p(P_n,P)=\|v_n\|_{L_p}^p\underset w \to \|\mathbb{V}\|_{L_p}^p$, recovering the conclusion of Theorem \ref{null_case_1d}. This can be made precise. For completeness we include next a formal statement of this fact.

\begin{theorem}\label{CLT_quantile_process}
Assume $P$ has a continuous density $f$ such that $f\circ F^{-1}$ is continuous in $(0,1)$ and monotonic in a neighbourhood of $0$ and $1$. Assume also that $J_p(P)<\infty$ for some  $p\geq 1$. Set
$$\mathbb{V}(t)=\frac{B(t)}{f(F^{-1}(t))},\quad 0<t<1,$$
with $B$ a Brownian bridge on $(0,1)$. Then, with probability one, the trajectories of $\mathbb{V}$ belong to $L_p(0,1)$. Furthermore,
$v_n\underset{w}{\to} \mathbb{V}$
as random elements in $L_p(0,1)$.
\end{theorem}
We include a schematic proof of this result in the supplementary material. We should note that the sufficient conditions here are not far from being necessary. In fact, in the recent paper \cite{beare2025necessarysufficientconditionsconvergence}, it is shown that, in the case $p=1$, weak convergence of $v_n$ in $L_1(0,1)$ holds if and only if $F^{-1}$ is absolutely continuous and  \eqref{assumption:L21} holds. Absolute continuity of $F^{-1}$ holds, in turn, if and only if $P$ is supported on an interval and the absolutely continuous component of $P$ has an a.e. positive density on that interval (this is Proposition A.17 in \cite{BobkovLedoux2019}). In the case $p=1$ it is of interest to observe that \eqref{assumption:L21} is the necessary and sufficient condition for weak convergence of $\sqrt{n}(F_n-F)$ as a random element in $L_1(\mathbb{R})$. Since $\sqrt{n}\mathcal{T}_1(P_n,P)=\|\sqrt{n}(F_n-F)\|_{L_1(\mathbb{R})}=\|v_n\|_{L_1(0,1)}
$, we see that weak convergence of $\sqrt{n}\mathcal{T}_1(P_n,P)$ holds in some cases in which $v_n$ fails to converge. 

{
\begin{remark}\label{remark_OT_map_1D} While looking at the quantile process leads, technically, to sharper limit theorems, the OT map is, arguably, the object of interest in applications. Furthermore, Theorem \ref{CLT_quantile_process} can be easily adapted to obtain distributional limit theorems for estimators of the OT map. The simplest adaptation would be to look at the optimal map from $Q$ to $P$, namely, $T=F^{-1}\circ G$, where we are assuming that $G$ is continuous. If we define $T_n:=F_n^{-1}\circ G$ then 
a trivial consequence of Theorem \ref{CLT_quantile_process} is that, under the same assumptions, the sequence $\sqrt{n}(T_n-T)$ converges
weakly to the centered Gaussian process
$$\mathbb{V}_G(x)=\frac{B(G(x))}{f(T(x))}, \quad x\in\mathbb{R},$$
as random elements in $L_p(\mathbb{R},Q)$. Slightly more elaborate arguments can be used to handle two-sample problems. CLTs in other metric spaces, which could be more appropriate in some applications, often require stronger assumptions. As a recent example we can cite \cite{Ponnopratetal2024}, which gives a uniform confidence band for $T$ in a compact interval $[a,b]\subset \mathbb{R}$, under additional smoothness assumptions on $F$ and $G$.
\end{remark}
}

\section{Limit theorems for general dimensions}\label{general_dimension}
Distributional limit theorems in the case $d>1$ are much more recent. We present now a {succinct} account of the main results. Remarkably, the CLTs for the fluctuations remain valid with very small changes. However, the natural centering constants cannot be used{,} apart from some particular cases. Given the interest of using the natural centering for inferential goals, we pay special attention to those cases. For a cleaner exposition we start dealing with the fluctuations, even though some of the CLTs that we present later were published earlier.

\subsection{CLTs for the fluctuation}

The following version of Theorem \ref{1dCLT} remains valid in general dimension:

\begin{theorem}\label{general1dCLT}
Assume $P,Q$ are probabilities on $\mathbb{R}^d$ with finite moments of order $2p$. Assume further that $P$ has a density and connected support with a negligible boundary. Then 
\begin{equation}\label{unbiased_CLT_general_d}
\sqrt{n}\left(\mathcal{T}_p(P_n,Q)-\mathbb{E}[\mathcal{T}_p(P_n,Q)]\right)\underset w \to N(0,\sigma_p^2(P,Q)),
\end{equation}
where 
\begin{eqnarray}
    \label{eq:Limitng-variance-multivariate}
    \sigma_p^2(P,Q)=\int f_0^2 dP-\left(\int f_0 dP\right)^2<\infty,
\end{eqnarray}
and $f_0$ is an OT potential from $P$ to $Q$ {for the cost $c(x,y)=\|x-y\|^p$}.
\end{theorem}

This is Corollary 4.7 in \cite{delBarrio.GonzalezSanz.Loubes.AIHP.2024.CLTgeneral}. Some comments are in order at this point. The assumptions in Theorem \ref{general1dCLT} guarantee that the OT potential $f_0$ is uniquely determined, up to the addition of a constant. Since $\sigma^2(P,Q)$ is the variance of the potential, this quantity does not depend on the particular choice of $f_0$ and $\sigma^2(P,Q)$ is well defined. The proof of the result includes the fact that $f_0\in L_2(P)$ if $P$ and $Q$ have finite moments of order $2p$. The apparently simpler description of the limiting variance in this result is explained by the stronger assumptions here. If $P$ is assumed to have a density in Theorem \ref{1dCLT} then $G^{-1}\circ F$ is the optimal map from $P$ to $Q$, the OT potential, $f_0$, is any primitive of the map $x\mapsto h_p'(x-G^{-1}(F(x)))$, with $h_p(x)=|x|^p$ and we see that the limiting variances in both results are equal. Finally, we would like to briefly sketch the path to prove \eqref{unbiased_CLT_general_d} through the \textit{Efron-Stein} linearization.  
Under sufficient integrability the (one-dimensional) CLT ensures that 
$$\sqrt{n}\left(\int f_0dP_n-\int f_0dP\right)\underset w \to N(0,\sigma_p^2(P,Q)).$$
We can use the Efron-Stein inequality to see that
\begin{multline*}
    n{\rm Var}\left[\mathcal{T}_p(P_n,Q)-\int f_0dP_n\right]\\
    \leq\, c \mathbb{E}\left[(f_n(X_1)-f_0(X_1))^2 \right],
\end{multline*}
for some constant $c>0$, where $f_n$ is an OT potential for the empirical problem --- an OT potential from  $P_n$ to $Q$. With a careful choice of centering constants the empirical OT potentials, $f_n(X_1)$, converge to $f_0(X_1)$ a.s.~and in $L_2$ (Theorem 2.10 in \cite{delBarrio.Loubes.2019.AoP.CLTquadratic}, Corollary 3.5 in \cite{delBarrio.GonzalezSanz.Loubes.AIHP.2024.CLTgeneral})  and we can conclude that the last upper bound vanishes as $n\to\infty$. But then \eqref{unbiased_CLT_general_d} follows. This scheme works provided $P$ and $Q$ have bounded support and in this case we also have convergence of variances
$$ n{\rm Var}\left[\mathcal{T}_p(P_n,Q)\right]\to \sigma_p^2(P,Q).$$
The moment assumption can be relaxed to finite moment of order $2p$, at the price of needing a more involved linearization technique and losing, possibly, the convergence of variances. We refer to \cite{delBarrio.GonzalezSanz.Loubes.AIHP.2024.CLTgeneral} for details.

\subsection{Centering constants in dimension $d>1$}
As in the one-dimensional case, it would be of interest to obtain conditions ensuring that the bias converges with  parametric rate towards zero as in \eqref{remove_bias_W2_one_sample}. {For $d \leq 3$, Theorem~\ref{theo:curse} allows us to replace $\mathbb{E}[\mathcal{T}_2(P_n, Q)]$ by the population cost $\mathcal{T}_2(P, Q)$ and conclude
\[
\sqrt{n} \big( \mathcal{T}_2(P_n, Q) - \mathcal{T}_2(P, Q) \big) \overset{w}{\longrightarrow} N(0, \sigma_2^2(P, Q)).
\]
This sort of result was considered by \cite{GonzlezDelgado.et.al.2023.EJS} for nonparametric goodness-of-fit testing on the flat torus. However, under the null hypothesis, the limit is degenerate; the correct convergence rate is faster (see again Theorem~\ref{theo:curse}). A more powerful test statistic could be constructed if a non-degenerate limiting distribution of $r_n \mathcal{T}_2(P_n, P)$, where $r_n$ converges to infinity slower than $\sqrt{n}$, is found (see \eqref{ledoux} below). 
}
For $d\geq 4$, in contrast to the one-dimensional case, integrability and smoothness assumptions do not suffice in all cases.  In fact, under the favorable situation in which $P$ and $Q$ are compactly supported on disjoint convex sets, the best possible general rate is that
$$\big|\E\left[\mathcal{T}_p(P_n,Q)\right] - \mathcal{T}_p(P,Q)\big|\leq c n^{-2/d},$$
see Corollary~3 in \cite{manole.weed.2024.AoAP}. This means that we cannot generally expect \eqref{remove_bias_W2_one_sample} to hold if $d\geq 4$.  We refer to \cite{niles2022estimation} for further details on the minimax estimation rates for $\mathcal{T}_p(P,Q)$. The main consequence is that estimation of $\mathcal{T}_p(P,Q)$ is affected by the curse of dimensionality and it is of interest to investigate particular setups in which parametric estimation is possible and to develop useful distributional limit theory in those setups. This is what we present next.

\subsection{Discrete spaces}

Discrete probabilities are a rich enough setup for many useful applications and an interesting benchmark to test alternative approaches to distributional limit theorems in OT. Let us assume that $P$ and $Q$ are supported on the finite sets $\mathbf{X}=\{x_1,\ldots,x_N\mathbb\}\subset \mathbb{R}^d$,
$\mathbf{Y}=\{y_1,\ldots,y_M\mathbb\}\subset \mathbb{R}^d$. Then $P$ is characterized by the probability $p_i=P(\{x_i\})$ given to each of the atoms of $\mathbf{X}$, i.e., $P=\sum_{i=1}^N p_i\delta_{x_i}$. Therefore, as the points are fixed,  the only variable is the element ${\bf p}=(p_1, \dots, p_N)$ of the $N$-dimensional simplex.
As a consequence, OT in discrete spaces is the finite-dimensional  linear program
\begin{equation}\label{simplex}
\min_{\pi\in \Gamma({\bf p},{\bf q})} \langle C, \pi\rangle_{{\rm F}},
\end{equation}
where \( C = (c(x_i,y_j)) \in \mathbb{R}^{N \times M} \) denotes the cost matrix and 
$$  \Gamma({\bf p},{\bf q})=\bigg\{\pi=(\pi_{ij})\in \R^{N\times M}:\ \sum_{j=1}^M \pi_{ij} = p_i, \, \sum_{i=1}^N \pi_{ij} = q_j \quad \pi_{ij} \geq 0\bigg\} $$
the transport polytope. The empirical measures  are defined by random elements $\hat{{\bf p}}=(\hat{p}_1, \dots, \hat{p}_N)$ on the simplex. Of course, there is no particular role of the ambient space $\mathbb{R}^d$ on the problem, but we try to keep a homogeneous notation. The multivariate central limit theorem shows that 
$\sqrt{n}(\hat{{\bf p}}-{{\bf p}})$
is asymptotically Gaussian with variance $\Sigma(\mathbf{p})$ defined as
\begin{equation}\label{sigma_Natrix}
   \begin{bmatrix}
p_1(1-p_1) & -p_1p_2  & \cdots& -p_1p_N\\
-p_2p_1 & p_2(1-p_2) & \cdots&-p_2p_N\\
\vdots & \vdots & \ddots&\vdots\\
-p_Np_1 &  -p_Np_2&\cdots&p_N(1-p_N)
\end{bmatrix}.
\end{equation}
Hence, in the discrete case, one can set $\mathcal{T}(\mathbf{p}):=\mathcal{T}_p(P,Q)$ and investigate  differentiability results for $\mathcal{T}$. The key result is that $\mathcal{T}$ is directionally Hadamard differentiable, with derivative
$$\mathbf{h}\mapsto \max_{\mathbf{u}\in\Phi^*_p(\mathbf{p},\mathbf{q})}\big(-\mathbf{u}\cdot\mathbf{h} \big),$$
where $\Phi^*_p(\mathbf{p},\mathbf{q})$ is the set of points $\mathbf{u}\in\mathbb{R}^N$ for which there exists $\mathbf{v}\in\mathbb{R}^M$ such that $\mathbf{p}\cdot \mathbf{u}+{\mathbf{q}}\cdot \mathbf{v}=\mathcal{T}_p(P,Q)$ and $u_i+v_j\leq \|x_i-y_j\|^p$, $1\leq i\leq N$, $1\leq j\leq M$ (these are the sets of optimal solutions to the dual of the linear program \eqref{simplex} when $c(x,y)=\|x-y\|^p$). Hence, the delta method shows the following. 
\begin{theorem}\label{CLT_discrete}
If $P$ and $Q$ are finitely supported probabilities on $\mathbb{R}^d$ then
\begin{equation}\label{discrete_CLT}
\sqrt{n}\big(\mathcal{T}_p(P_n,Q)-\mathcal{T}_p(P,Q)\big)\underset w \to \max_{\mathbf{u}\in\Phi^*_p(\mathbf{p},\mathbf{q})} \mathbf{G}\cdot \mathbf{u}, 
\end{equation}
where $\mathbf{G}$ is a centered Gaussian r.v. with covariance matrix $\Sigma(\mathbf{p})$ as in \eqref{sigma_Natrix}.
\end{theorem}
Theorem \ref{CLT_discrete} is a simplified version of Theorem 3 in \cite{SommerfeldMunk2018}. The case $P=Q$ admits a slightly cleaner expression for the limit, generalizing earlier related work \cite{SamworthJohnson2004}. We see that normality of the limit holds if the maximizer in the dual problem is unique, but, obviously, this is not directly related to the smoothness of the cost. A similar result holds when the supports of $P$ and $Q$ are countable, under the additional assumption 
$$\sum_{i=1}^{\infty} \|x_i\|^p \sqrt{p_i}<\infty,$$
see \cite{TamelingSommerfeldMunk2019} for details. 

A central limit theorem  for the OT plans in discrete spaces is derived in \cite{Klatt.et.al.2022.AoOR}, where  the corresponding limit distribution is usually non-Gaussian and is  determined from
the way the empirical plan is chosen. The specifics of the statement are highly technical and extend beyond the scope of this discussion. For a comprehensive treatment and the formal statement, we direct the reader to \cite[Theorem~6.1]{Klatt.et.al.2022.AoOR}. { Recent work (see \cite{liu.2025entropicregularizationdebiasedgaussian}) has established estimators for the discrete OT plan that converge to a Gaussian limiting distribution. This result is achieved using OT plans derived from regularization with non-entropic penalizations and letting the regularization parameter tend to zero.}

\subsection{Lower-complexity adaptation principle}
\label{sec:lower_adapt}
Between the dual and primal formulation of OT we have the semidual formulation 
\begin{align}\label{eq:Semidual}
\mathcal{T}_c(P,Q)={\max_{f \in \mathcal F}}\int f({x})dP(x)+\int f^{c}(  {y}) dQ(  {y}),
\end{align}
where {$\mathcal F$ is a suitable class of functions, and}
$$ f^{c} = \inf_{x\in \R^d} \{ c(x,y)-f(x)\}$$
denotes the $c$-conjugate of $f$. 

{Then, from the bound 
\begin{equation}
\label{eq:semidual_emprical_process}
\left|\mathcal T_c(P_n,Q)-\mathcal T_c(P,Q)\right|
 \leq
\sup_{f\in \mathcal F}\left|(P_n-P)f\right|, 
\end{equation}
it is clear that the complexity of the class $\mathcal F$ plays a fundamental role on the convergence rates of the cost.
}

{It was first observed by \cite{Forrow.et.al.AiStats.19} that if one of the marginals is discrete, then the OT cost does not suffer from the curse of dimensionality. } {In the} seminal work  \cite{Hundrieser.et.al.2024.AIHP}, the authors realized that $c$-conjugate operations {does not increase} the complexity of the class, i.e., the covering numbers of 
{
\(\mathcal F^c=\{f^c:f\in\mathcal F\}\) are bounded by those of
\(\mathcal F\), for any bounded cost}. This clever observation implies that the statistical complexity of the OT cost adapts to the smaller of the complexities of the probability measures as we explain below.

If the cost function is  Lipschitz, with constant $L$, and both measures are supported on compact sets, then the empirical and population OT potentials $f_*$  and $f_n$  belong to the class ${\mathcal F=}{\rm Lip}_L({\rm supp}(P))  $ of $L$-Lipschitz functions over the support ${\rm supp}(P)$ of $P$. If ${\rm supp}(P)$ has low intrinsic dimension then ${\rm Lip}_L({\rm supp}(P))$ 
{has small covering number. If ${\rm supp}(Q)$ has low intrinsic dimension, then, by symmetry, one can select $\mathcal F =({\rm Lip}_L({\rm supp}(Q)))^c$, which, by the previous discussion, again has small covering numbers.} This implies the following result (for a more general statement we refer to \cite[Theorems~3.3 and 3.8]{Hundrieser.et.al.2024.AIHP}). Here the complexity of the support ${\rm supp}(\mu)$ of a measure $\mu$ is quantified by its covering numbers $ \mathcal{N}(\epsilon, {\rm supp}(\mu), \|\cdot\|)$. 
\begin{theorem}\label{theorem:Shayan}
    Let $P$ and $Q$ be compactly supported probability measures on $\R^d$. Assume that there exist $C,\epsilon_0,\kappa>0$ such that 
    $$ \mathcal{N}(\epsilon, {\rm supp}(\mu), \|\cdot\|) \leq C \epsilon^{-\kappa} \quad \text{for } \epsilon\leq \epsilon_0$$
   and $\mu=P$ or $\mu=Q$. Then 
$$ \E\left[\left|  \mathcal{T}_{p}(P_n,Q)-\mathcal{T}_{p}(P,Q)\right|\right]\lesssim  a_{n,\kappa,p}, $$
where 
$$ a_{n,\kappa,p}=\begin{cases}
    n^{-\frac{1}{2}}& \hspace{-1mm}\text{if  $\kappa< 2$ or $\kappa\leq  3$ and $p>1$} \\
    \frac{\log(n)}{n^{\frac{1}{2}}}& \hspace{-1mm}\text{if  $(\kappa,p)=(2,1)$ or $\kappa=4$ and $p>1$}\\
    n^{-\frac{1}{\kappa}}& \hspace{-1mm}\text{otherwise}. 
\end{cases}$$
\end{theorem}
Theorem~\ref{theorem:Shayan} implies that if one of the measures has  intrinsic dimension $\kappa\leq 3$ then the OT cost is not cursed by dimensionality for  $p>1$. Moreover, the optimization class in the dual formulation \eqref{eq:potentials} is intrinsically Donsker. Then, by means of the functional delta method, we obtain the following result (again, this is a simplified version of the results in \cite{Hundrieser.et.al.BJ.2024.Unifying}, which should be consulted for a much more complete account of the implications of this lower complexity adaptation principle).

\begin{theorem}\label{theorem:Shayan:CLT}
Assume $P,Q$ are compactly supported probabilities on $\mathbb{R}^d$ with finite moments of order $2p$, where $p\geq 2$ and $d\leq 3$. Assume further that $P$ has a density and connected support with a negligible boundary. Then 
\begin{equation}\label{biased_CLT_general_d}
\sqrt{n}(\mathcal{T}_p(P_n,Q)-\mathcal{T}_p(P,Q))\underset w \to N(0,\sigma_p^2(P,Q)),
\end{equation}
{where 
$\sigma_p^2(P,Q)$ 
is as in \eqref{eq:Limitng-variance-multivariate}.}
\end{theorem}
The assumptions in Theorem \ref{theorem:Shayan:CLT} guarantee uniqueness (up to an additive constant) of the OT potential. This is the key to the Gaussianity of the limiting distribution, but other non-Gaussian limits can be obtained without this assumption. 

Now that we have established that the statistical complexity of the transport cost adapts to the lower intrinsic dimension of the marginal probability measures, a natural question arises: does the same hold for the transport potentials or plans? This remains an open problem, except in the semi-discrete case, which we will examine next.

\subsection{Semi-discrete OT}

An important application of Theorem~\ref{theorem:Shayan} arises in the setting of semi-discrete optimal transport, where one of the probability measures, $P$ or $Q$, is discrete while the other remains continuous. In this case, the entropy  numbers of the discrete measure’s support are inherently bounded, ensuring that the CLT for the empirical OT cost holds irrespective of the dimension. The next result was established by \cite{delBarrioetAl.2024.BJ}, independently of the related findings in \cite{Hundrieser.et.al.2024.AIHP,Hundrieser.et.al.BJ.2024.Unifying}. 
\begin{theorem}\label{semidiscrete_CLT_cost}
Assume $P, Q$ are probabilities on $\mathbb{R}^d$ with $P$ finitely supported and $Q$ having a density and connected support with a negligible boundary. If $Q$ has finite moment of order $p$ then 
\begin{equation}\label{semidiscrete_CLT_1}
\sqrt{n}(\mathcal{T}_p(P_n,Q)-\mathcal{T}_p(P,Q))\underset w \to N(0,\sigma_p^2(P,Q)),
\end{equation}
{where 
$\sigma_p^2(P,Q)$ 
is as in \eqref{eq:Limitng-variance-multivariate}.}
The same result holds if $Q$ is finitely supported, 
and $P$ has a density with connected support and negligible boundary and finite moment of order $2p$. 
\end{theorem}

In this semidiscrete setup it is possible to give CLTs for the optimal potentials, as we discuss next. In our presentation we focus only on the squared-Euclidean cost $c_2$. We assume that  $P$ is a discrete probability measure with atoms $x_1, \dots, x_N\in \R^d$ and weights $p_1, \dots, p_N$ in the $N$-dimensional simplex, i.e., $P=\sum_{i=1}^N p_i \delta_{x_i}$. The semidual formulation of this problem maximizes the function $\mathbb{M}(z)$ defined as
\begin{equation}
    \label{semi-dual-intro}
\sum_{i=1}^N z_ip_i+\int \min_{i=1, \dots, N} \{ \|x_i- y \|^2 -z_i \}dQ(y),
\end{equation}
where the solution $z^* $ of \eqref{semi-dual-intro} encodes the OT potentials and defines the Laguerre cells 
$$    \operatorname{Lag}_k(z):=\{ y :\ \ \|x_k-y \|^2 -z_k^* <\|x_i-y \|^2-z_i^*, \  \forall\, i\neq k \}.$$
 Semi-discrete optimal transport plays an important role in applications in economic modeling, quantization \cite{Siegfried.Luschgy.2000.Book},  partial differential equations \cite{Gallouet.Merigot.2018.Found.Comp.Math}, seismic imaging \cite{Meyron.2019.Compt.Aid.Des} or astronomy \cite{Levy.et.al.2021.Royal.Astron.Soc}. For instance, in quantization the Laguerre cells provide the optimal assignment of data points   $Q$ to the centroids $x_1, \dots, x_N$. In economic applications (see \cite{Galichon.2016.Book}), {the} discrete set $\{x_1, \dots, x_N\}$ represents {the location of different sellers of a common product,}   weights $p_1, \dots, p_N$ represent the capacity of each seller, $Q$ the distribution of consumers, and the cost function $c$ represents minus the utility function $U$.  The sellers aim at raising the prices,  represented by the vector $(v_1, \dots, v_N)$, and a consumer $y$ chooses a store according to the following {criterion};   maximize the utility while minimizing the price, i.e.,   
 $$ \argmin_{i=1, \dots, N} \{ \|x_i- y \|^2 +v_i \}. $$
 The equilibrium of the system is found when the prices are the OT potentials for the cost function equal to minus the squared-Euclidean distance, see \cite{Galichon.2016.Book}. Hence,  Laguerre cells represent the demand sets of the market and the results obtained by \cite{delBarrioetAl.2024.BJ,Sadhu.Goldfeld.Kato.2024.AoAP.Samidi}, which we describe below, provide confidence bands to the optimal prices and demand sets of random markets. 
 
\subsubsection{Potentials}
In \cite{Kitagawa.Merigot.Thibert.2019.JEMS} it was shown that if $Q$  is supported on a compact convex set $\Omega'$ and has a continuous density bounded away from zero,  the  functional  \eqref{semi-dual-intro} is concave and twice differentiable with first derivative 
    \begin{equation}
      \label{derivativeZest}(-Q(\operatorname{Lag}_1(z))+p_1, \dots, -Q(\operatorname{Lag}_N(z))+p_N ),
    \end{equation}
      and Hessian matrix $\nabla^2\mathbb{M}(z)=\left( \frac{\partial^2}{\partial z_i \partial z_j }  \mathbb{M}(z)\right)_{i,j=1,\dots,N}$  with partial derivatives 
\begin{align*}
    \frac{\partial^2}{\partial z_i \partial z_j }  \mathbb{M}(z)&=\int_{\operatorname{Lag}_i(z)\cap \operatorname{Lag}_j(z) } \frac{q(\mathbf{y})}{\|x_i-x_j\|}d\mathcal{H}^{d-1}(y),\\
    \end{align*}
    if $i\neq j$,
    and
    \begin{align*}
\frac{\partial^2}{\partial^2 z_i} \mathbb{M}(z) &=-\sum_{j\neq i}\frac{\partial^2 }{\partial z_i\partial z_j} \mathbb{M}(z).  
\end{align*}
Moreover, under these conditions the Hessian $\nabla^2\mathbb{M}(z^*)$ is a  negative symmetric matrix having zero as simple eigenvalue, corresponding to the  `constant' vector $(1, \dots, 1)$.  
Hence, the optimal potential $z^*$ is then found where the equilibrium 
$Q(\operatorname{Lag}_i(z^*))=p_i  $ is attained, for all $i=1, \dots, N$. The nondegeneracy of the Hessian yields a Polyak-Lojasiewicz inequality for the dual functional. Then, due to the intrinsic parametric complexity of the problem, \cite{delBarrioetAl.2024.BJ} proved that the rates of convergence of the empirical potential $z^{(n)}$   are parametric and derived the following central limit theorem. Some of the assumptions have been relaxed by \cite{Goldfeld.et.al.2024.AoAP}.
\begin{theorem}\label{theorem:barrioSemi}
    Let $Q$ be supported on a compact convex set in which it has a continuous density bounded away from zero. Then it holds that 
    $$ \sqrt{n}(z^{(n)}-z^*){ \underset w \to } \left(\nabla^2\mathbb{M}(z^*)\right)^{-1}(\mathbb{U}_1, \dots, \mathbb{U}_N),$$
    where $(\mathbb{U}_1, \dots, \mathbb{U}_N)$ is a centered multivariate Gaussian with variance \eqref{sigma_Natrix}. 
\end{theorem}
The other potential can be derived 
from $z^*$ via the relation 
$$ f_*(y)= \min_{i=1, \dots, N} \{ \|x_i- y \|^2-z_i^*\} . $$
The empirical potential $f_n(y)$ is obtained in a similar fashion.
The observation 
\begin{align*}
     \|f_n-f_*\|_\infty&=\sup_{y\in \Omega'} |f_n(y)-f_*(y)|\\ &= \max_i|z_i^{(n)}-z_i^*|=\|z^{(n)}-z^*\|_\infty,
\end{align*}
and Theorem~\ref{theorem:barrioSemi} yield the limit
$$ \sqrt{n} \|f_n-f_*\|_\infty { \underset w \to }\big\|(\nabla^2\mathbb{M}_{P}(z^*))^{-1}((\mathbb{U}_1,\dots, \mathbb{U}_N))\big\|_{\infty}.$$
Laguerre cells are another interesting object related to the semi-discrete optimal transport. They are set-valued mappings of the potentials $z^*$. Since each cell is convex, we can measure the distances between the cells via the support functions. Since the cells,  as defined above, are not bounded, we truncate them in a sufficiently large ball containing the support of $Q$ and define 
$$ \operatorname{Lag}_i^R(z)= \overline{\operatorname{Lag}_i(z)\cap R\, \mathbb{B}_d},$$
{where $\overline{A}$ denotes the closure of a set $A$.}

Recall that the support function of a bounded convex set $S$ is the convex conjugate of the convex indicator function of $S$, i.e.,  $h_S(x)=\sup_{y\in S} \langle x,y\rangle$. The values in the unit sphere $\mathcal{S}^{d-1}$ of the support function characterize a bounded convex set.  Hence, for $p\in [1, \infty)$ we define $L^p$ metric  as $$d_p(A,B):=\left(\int_{\mathcal{S}^{d-1}} |h_{A}-h_{B}|^p d\mathcal{H}^{d-1}\right)^{\frac{1}{p}},$$
where $\mathcal{H}^{d-1}$ is the Hausdorff measure in $\mathcal{S}^{d-1}$,   and the uniform norm  $ d_{\infty}(A,B)=\sup_{v\in \mathbb{S}^{d-1}} |h_{A} -h_{B} |, $ which corresponds
{to} the Hausdorff distance between compact convex sets. Using the characterization
$$   h_{ \operatorname{Lag}_k^R(z)}(v)
   =\min_{t_j>0}\Big\lbrace \sum_{j\neq k}t_j (\| {x}_k\|^2-\| {x}_j\|^2-z_k+z_j)
   +R\bigg\|v-\sum_{j\neq k}t_j (\mathbf{x}_k-\mathbf{x}_j) \bigg\|\Big\rbrace$$
of the support functions of the Laguerre cells, \cite{delBarrioetAl.2024.BJ} proved the following limit theorem. 
\begin{theorem}\label{theorem:barrioSemi-Lag}
    Let $Q$ be supported on a compact convex set, with a continuous density bounded away from zero. Then the sequence,  
    $$ \sqrt{n}\left( h_{ \operatorname{Lag}_k^R(z^n)}-h_{ \operatorname{Lag}_k^R(z^*)}\right)$$ has a Gaussian limit in distribution in $L^p$ for all $p\in [1, \infty)$.
\end{theorem}
While no CLT in Hausdorff distance is available for the Laguerre cells, \cite{delBarrioetAl.2024.BJ}  
provides confidence intervals for the Hausdorff distance between the cells (Remark~4.7).  
\subsubsection{Plans and maps}
Limit theorems of the $L^p$-distances between the support functions of the cells are useful to provide confidence bands on the demand sets in economic applications. However, these distances do not take into account the element of the discrete set $\{x_1, \dots, x_N\}$ associated with the cell, which  is useful in  some applications such as quantization. This can be {measured} by means of the transport map, i.e., the a.s.\ defined gradient of the convex function  $\varphi_n(y)=\|y\|^2/2- g_n(y)/2$, {where $g_n$ is an empirical OT potential.} { Here the most popular distance  is that induced by the ${L^{\gamma}(Q)}$ distance for $\gamma\geq 1$.} Unfortunately, as \cite{Sadhu.Goldfeld.Kato.2024.AoAP.Samidi} proved, the OT maps do not satisfy a CLT in  ${L^{\gamma}(Q)}$. However, they do satisfy a central limit theorem in dual topologies. We denote by $\mathcal{C}^{\beta}(\Omega';\R^d)$  the Banach space of continuous $\beta$-H\"older functions and by $(\mathcal{C}^{\beta}(\Omega';\R^d))'$ its  dual space. 
\begin{theorem}\label{theorem:Kato}
   Let $Q$  be supported on a compact convex set and have a continuous density bounded away from zero. Then 
    \begin{enumerate}
        \item There exists a nonzero random variable $\mathbb{V}_\gamma$ such that  $n^{\frac{1}{2\gamma}}\|\nabla \varphi_n -\nabla \varphi_*\|_{L^{\gamma}(Q)}  { \underset w \to } \mathbb{V}_\gamma $;
        \item There exists a tight random element $\mathbb{G}$ in the dual Banach space $(\mathcal{C}^{\beta}(\Omega';\R^d))'$ such that the element
        $$  \mathcal{C}^{\beta}(\Omega';\R^d)\ni f\mapsto  \sqrt{n}\langle \nabla \varphi_n -\nabla \varphi_*, f\rangle_{L^{2}(Q)}$$
        of $(\mathcal{C}^{\beta}(\Omega';\R^d))'$ converges in distribution to $\mathbb{G}$ in  $(\mathcal{C}^{\beta}(\Omega';\R^d))'$.
    \end{enumerate}
\end{theorem}
Apart from these two points, \cite{Goldfeld.et.al.2024.AoAP} proved the central limit theorem of $\sqrt{n}\langle \nabla \varphi_n -\nabla \varphi_*, f\rangle_{L^{2}(Q)}$ for a fixed test function $f$ with negligible discontinuities w.r.t.\ the $d-1$-dimensional Hausdorff measure.  
Proposition~4 in \cite{Sadhu.Goldfeld.Kato.2024.AoAP.Samidi} uses this result to show that $r_n(\nabla \varphi_n -\nabla \varphi_*)$ does not possess nonzero distributional limits in $L^2(Q)$. Their elegant and short proof uses the Hilbertian structure of $L^2(Q)$. Inspired by that, we derive the same condition in any $L^{\gamma}(Q) $ space for $\gamma\in (1, \infty)$. Assume that $ r_n(\nabla \varphi_n -\nabla \varphi_* )$
has a nonzero limit $\mathbb{U}$ in $L^{\gamma}(Q) $, then $r_n=n^{\frac{1}{\gamma}}$ by the first point of Theorem~\ref{theorem:Kato}.
The continuous mapping theorem yields 
$$ \sup_{\|f\|_{\mathcal{C}^{\beta}(\Omega';\R^d)}\leq 1}n^{\frac{1}{\gamma}}\langle \nabla \varphi_n -\nabla \varphi_*, f\rangle_{L^{2}(Q)} 
    { \underset w \longrightarrow } \sup_{\|f\|_{\mathcal{C}^{\beta}(\Omega';\R^d)}\leq 1}\langle \mathbb{U}, f\rangle_{L^{2}(Q)}, $$
so that, by the second point of Theorem~\ref{theorem:Kato}, we get  $$ \mathbb{P}\left(\langle \mathbb{U}, f\rangle_{L^{2}(Q)}=0, \ \forall f\in \mathcal{C}^{\beta}(\Omega';\R^d)\right)=1. $$
Since $\mathcal{C}^{\beta}(\Omega';\R^d)$ is dense in  $L^{\gamma^*}(Q)$, where $\gamma^*\in (1,\infty)$ is the conjugate exponent of $\gamma$, we get $\mathbb{U}=0 $ in $L^{\gamma}(Q) $.

Let us delve into the details of the proof of \cite{Sadhu.Goldfeld.Kato.2024.AoAP.Samidi}. First, {the reason for these mismatched
rates} among the different $L^\gamma$ norms is due to the relation 
$$ \|\nabla \varphi_n -\nabla \varphi_*\|_{L^{\gamma}(Q)} =\left( \sum_{i,j=1}^K \|x_i-x_j\|^{\gamma}  Q\left({\rm Lag}_i(z^{*}) \cap  {\rm Lag}_j(z^{(n)})\right) \right)^{\frac{1}{\gamma}}, $$
which  implies that  $\|\nabla \varphi_n -\nabla \varphi_*\|_{L^{\gamma}(Q)}$ has the same convergence rate as
$$ \left( \sum_{i\neq j}^K Q\left({\rm Lag}_i(z^{*}) \cap  {\rm Lag}_j(z^{(n)})\right) \right)^{\frac{1}{\gamma}}. $$
Hence, it depends on $\gamma$. Moreover, the limit theorems follow from that of $$Q\left({\rm Lag}_i(z^{*}) \cap  {\rm Lag}_j(z^{(n)})\right).$$ Therefore, the relation
$$  \langle \nabla \varphi_n -\nabla \varphi_*, f\rangle_{L^{2}(Q)} 
    = \sum_{i,j=1}^K \int_{{\rm Lag}_i(z^{*}) \cap  {\rm Lag}_j(z^{(n)})}\langle x_i-x_j, f(y) \rangle  dQ(y), $$
indicates that the limit of the family of  indicator functions $ \mathbb{I}\left[{\rm Lag}_i(z^{*}) \cap  {\rm Lag}_j(z^{(n)})\right]$ integrated w.r.t.\  test functions determines all the limits of Theorem~\ref{theorem:Kato}. The strategy of \cite{Goldfeld.et.al.2024.AoAP} is to show the differentiability of the functional 
$$ z\mapsto  \int_{{\rm Lag}_i(z^{*}) \cap  {\rm Lag}_j(z)}f(y)   dQ(y),$$
at $z^*$  in order to apply the delta-method and Theorem~\ref{theorem:barrioSemi}.

\section{Regularized problems}\label{Section:regularized}
Beyond the cases described in Section~\ref{general_dimension}, estimation of OT (cost, potentials or plans) is affected by the curse of dimensionality. As noted in the Introduction, this issue and computational considerations have motivated the adoption of regularized versions of OT. We devote this section to distributional limit theory in this setup.
\subsection{Entropy regularized optimal transport}
The entropy regularized optimal transport (EOT),  popularized by \cite{Cuturi2013SinkhornDL}, corresponds to  the minimization of the OT cost together with a penalty
\begin{equation}\label{EOT}
     {\cal E}_\epsilon(P,Q)=\min_{\pi \in \Pi(P,Q)}\int_{\R^d\times \R^d} c(x,y) d \pi(  {x},   {y}) +
\epsilon \cdot {\rm H}(\pi\vert P\otimes Q),
\end{equation}
where $ {\rm H}(\pi\vert P\otimes Q)$ denotes the Kullback-Leibler divergence of $\pi$ w.r.t.\ the product probability measure $P\otimes Q$ and  $\epsilon > 0$ balances the weight of the penalty.  The minimizer of \eqref{EOT} is called the Entropic Optimal Transport (EOT) plan while the maximizers $(f_{\epsilon}, g_{\epsilon})$ {in $L^2(P)\times L^2(Q)$}   of the dual functional
\begin{equation} 
\nonumber
\label{ENGdual_entropESP}
 \int \Big\{ f(x) + g(y) 
	-\epsilon \cdot e^{\frac{f(x) + g(y)- c(x,y) }{\epsilon} }+\epsilon\Big\} d (P\otimes Q)(x,y)
\end{equation}
are called EOT potentials. They can be equivalently defined 
as the unique solutions of the Schrödinger system
\begin{align}
\begin{split}
&\int \exp{\Big(\frac{f_{\epsilon}(x)+g_{\epsilon}(y)- c(x,y)}{\epsilon}\Big)}dQ(y)  =1, \ \ \text{$x\in \R^d $},\\
    \label{SrodingerSystem}
	&\int \exp{\Big(\frac{f_{\epsilon}(x)+g_{\epsilon}(y)-c(x,y)}{\epsilon}\Big)}dP(x)  =1, \ \ \text{ $y\in \R^d $}.
\end{split}
\end{align}
EOT is undoubtedly the most popular among all the variations  of OT, mainly due to three factors, (i) the Sinkhorn algorithm \cite{Sinkhorn.1967.AMM} is an iterative fixed point algorithm with linear convergence rates \cite{FRANKLIN.1989.LAA,Carlier.2022.SIOPT} which allows for the efficient computation  of EOT plans, potentials and costs \cite{Cuturi2013SinkhornDL};  (ii) its connections with the Schrödinger bridge \cite{Schrodinger.1932.AIHP} and the theory of large deviations \cite{Leonard.2014.DCDS}  and (iii) it avoids the curse of dimensionality for a fixed regularization parameter $\epsilon>0$ \cite{genevay.et.al.2019.PMLR,Mena.Weed.2019.Nips}.  {
We should also mention here the recent preprint~\cite{mordant.2024.Preprint}, which studies central limit theorems in EOT in the regime where the regularization parameter tends to zero.
}

{Recently, other types of entropy penalties for the OT problem have been proposed in the literature \cite{Muzellec.2017.AAAI,blondel18quadratic,Nutz.24}. These approaches may be particularly applicable in cases where a sparse approximation of the OT plan is desired \cite{GonzalezSanzNutz2024.Sacalar,GonzalezSanzNutzRivero2024.Monotone,zhang.2023.manifoldlearningsparseregularised}.}

{
We do not cover the details of these variations in this survey. However, it is worth mentioning that the statistical complexity was first studied in \cite{Bayraktar.Eckstein.2025.BJ} through the control of dual potential smoothness. Following this approach, \cite{Bayraktar.Eckstein.2025.BJ} obtained convergence rates that suffer from the curse of dimensionality. In contrast, \cite{GonzalezNutzEcsktein.2025, gonzalezsanz2025samplecomplexityquadraticallyregularized} demonstrated, using  different proof techniques, that regularized transport problems do not generally suffer from the curse of dimensionality.}

\subsubsection{Costs}
Proving a CLT requires a control on the regularity of the EOT potentials $(f_{\epsilon}, g_{\epsilon})$. From Equation~\eqref{SrodingerSystem}, we first observe that they inherit the regularity of the cost function. Hence, for the squared-Euclidean cost, the EOT potentials are $\mathcal{C}^\infty$. Moreover, for $P$ and $Q$ supported on compact sets $\Omega$ and $\Omega'$,  the derivatives of $(f_{\epsilon}, g_{\epsilon})$ are bounded by a constant depending on the diameters of $\Omega$, $\Omega'$ and the regularization parameter $\epsilon$ (subgaussian tails yield the same result, cf.~\cite{genevay.et.al.2019.PMLR,Mena.Weed.2019.Nips}). That is, both the empirical $(f_{\epsilon,n}, g_{\epsilon,n})$ and population $(f_{\epsilon}, g_{\epsilon})$  EOT potentials belong to a ball in $\mathcal{C}^s(\Omega)\times\mathcal{C}^s(\Omega') $ of radius independent of the sample size, for all $s\in \N$. Hence the random sequence  $(f_{\epsilon,n}, g_{\epsilon,n})$ belongs to a Donsker class. In  \cite{Mena.Weed.2019.Nips,genevay.et.al.2019.PMLR} these estimates have been used to show that 
$$ \E[| {\cal E}_\epsilon(P_n,Q)-{\cal E}_\epsilon(P,Q)|] \leq   C\, n^{-\frac{1}{2}}$$
for smooth costs.

In the same work, it is shown that in the bias-variance decomposition of ${\cal E}_\epsilon(P_n,Q)- {\cal E}_\epsilon(P,Q)$, the variance term ${\cal E}_\epsilon(P_n,Q) -\E[{\cal E}_\epsilon(P_n,Q)]$ can be handled with the Efron-Stein linearization of \cite{delBarrio.Loubes.2019.AoP.CLTquadratic}, leading to  
$$ \sqrt{n}({\cal E}_\epsilon(P_n,Q) -\E[{\cal E}_\epsilon(P_n,Q)]) { \underset w \to } \mathcal{N}(0,{\rm Var}_P(f_\epsilon)).$$
Additionally, \cite{delBarrio.et.al.2023.SIMODS.Improved} showed that the bias term $|\E[{\cal E}_\epsilon(P_n,Q)]-{\cal E}_\epsilon(P,Q)|$ converges to zero faster than the variance, again with arguments based on the uniform smoothness of the potentials. As a consequence, the following CLT holds.
\begin{theorem}\label{sinkhorn_CLT}
Assume that $P$ and $Q$ are subgaussian probabilities on $\mathbb{R}^d$ and $c(x,y)=\|x-y\|^2$. Then 
$$\sqrt{n}({\cal E}_\epsilon(P_n,Q) -{\cal E}_\epsilon(P,Q)) { \underset w \to }  \mathcal{N}(0,{\rm Var}_P(f_\epsilon)),$$
where $f_\epsilon$ is an  EOT potential.
\end{theorem}
A different approach to obtain the same result  was proposed by \cite{Goldfeld.et.al.2022.statisticalinferenceregularizedoptimal} using  that  $(f_{\epsilon,n}, g_{\epsilon,n})$ belongs to a Donsker class and employing the functional Delta method, yet under the restrictive assumption that the distributions are compactly supported.  In the discrete setting the same technique also provides the central limit theorems, cf.~\cite{klatt.et.al.2020.SIMODS,Hundrieser.et.al2021.AoAP}. This last technique, even if effective for dealing with a wider range of problems, does not provide, in general, the sharpest conditions as  illustrated later by \cite{Rigollet.Strome.2022.AOSappear}. Avoiding the use of empirical processes theory and exploiting the strong convexity of the dual formulation of EOT,  \cite{Rigollet.Strome.2022.AOSappear} derived the same bound of the bias term as \cite{delBarrio.et.al.2023.SIMODS.Improved}.  {This approach}  was used by \cite{GonzalezSanz.Hundrieser..2023.Beyond} to derive the CLT for the EOT with general cost functions under more general assumptions, as follows. 
\begin{theorem}
    Assume that the cost function $c:\R^d\times \R^d\to \R$ is bounded and measurable in the support of $P\otimes Q$. Then it holds that 
    $$ \sqrt{n}({\cal E}_\epsilon(P_n,Q) -{\cal E}_\epsilon(P,Q)) { \underset w \to } \mathcal{N}(0,{\rm Var}_P(f_\epsilon)).$$
\end{theorem}
\subsubsection{Potentials and Plans}
As stated previously, the empirical EOT potentials are the unique solutions (up to additive constants) of the empirical version of the Schrödinger system \eqref{SrodingerSystem}, that is, the version in which we replace $P$ with $P_n$. It is convenient at this point to introduce the notation
$$ \Gamma(f,g):=\Big(\int \exp{\Big(\frac{f(x)+g(y)- c(x,y)}{\epsilon}\Big)}dQ(y),\int \exp{\Big(\frac{f(x)+g(y)-c(x,y)}{\epsilon}\Big)}dP(x)\Big),$$
and, similarly, we write $\Gamma_n$ for the empirical version.
The empirical EOT potentials are then defined through the equation $\Gamma_n(f_{\epsilon,n},g_{\epsilon,n})=(1,1)$. This definition enables to reframe the estimation of EOT potentials as a  $Z$-estimation problem in the nomenclature of empirical processes. In this framework, CLTs are {obtained} using the implicit function theorem in Banach spaces.  This approach was employed by \cite{GonzalezSanz.Loubes.Weed.2024.weaklimits,Goldfeld.et.al.2024.EJS}  to prove a CLT for the EOT potentials for $\mathcal{C}^\infty$ costs.  The derivative of the Schrödinger operator $\Gamma$ is of the form  $\mathbb{L}=I+\mathcal{A}$ where $\mathcal{A}$ is a compact operator in $\mathcal{C}(\Omega) \times \mathcal{C}(\Omega') $. The first nontrivial eigenvalue of $\mathcal{A}$ is $-1$, which has multiplicity one and is generated by the constant function $(1,-1) $. Any other eigenvalue of  $\mathcal{A}$ is bounded away from one by Jensen's inequality. Hence, the Fredholm alternative yields the invertibility of the linearized operator $\mathbb{L}$ (see \cite{Carlier.Laborde.SIMA.2020}).\\
Actually,  their approach can be extended to Lipschitz costs by means of the following observation.  
Up to a term of order
$  o_{\mathbb{P}}(\|f_\epsilon\oplus g_\epsilon-f_{\epsilon,n}\oplus g_{\epsilon,n}\|_\infty), $
the expansion 
$$ \Gamma(f_\epsilon,g_\epsilon)-\Gamma(f_{\epsilon,n},g_{\epsilon,n}) = \mathbb{L}(f_\epsilon-f_{\epsilon,n},g_\epsilon-g_{\epsilon,n}) $$
holds. Now, proving a CLT for the left-hand side of the previous equation is enough to find the limiting distribution of the EOT potentials. This follows from the fact that the random function    $ e^{\frac{-c(X_i,\cdot)}{\epsilon}}$ satisfies a CLT in $\mathcal{C}(\Omega)$ for any Lipschitz cost (see \cite{LedouxTalagrand1991} for a discussion on the conditions where a uniform central limit theorem holds). We refer to \cite{GonzalezSanz.Loubes.Weed.2024.weaklimits} for the description of the covariance of the Gaussian process in the following theorem.  
\begin{theorem}\label{theorem:CLT-Potential-EOT}
    Let $P$ and $Q$ be supported on the compact sets $\Omega$  and $\Omega'$. Assume that the cost function is Lipschitz. Then there exists a  centered Gaussian process $\mathbb{G} $  in $\Omega \times\Omega'  $ with continuous sample paths such that  
    $$ \sqrt{n}(f_\epsilon-f_{\epsilon,n},g_\epsilon-g_{\epsilon,n}){ \underset w \to } \mathbb{G} $$
    in $\mathcal{C}(\Omega \times\Omega')$. If the cost is $\mathcal{C}^s$ then the limit holds in $\mathcal{C}^{s-1}(\Omega \times\Omega')$. 
\end{theorem}
From this result a CLT for EOT plans can be established through the asymptotic analysis of the dual potentials and the  relation
\[
d\pi_\epsilon(x,y) = \exp\Big({\frac{f_{\epsilon}(x) + g_{\epsilon}(y) - c(x,y)}{\epsilon}}\Big) \, dP(x) \, dQ(y)
\]
and its empirical counterpart.
The first work showing central limit theorems for EOT plans was   \cite{Harchaoui.Liu.Pal.2020.BJ.} for a different estimator,     where the authors conjectured the validity of the result for the plug-in estimator $\pi_{n,\epsilon}$ and any bounded cost function  (cf.~Remark~1, ibid). 
In \cite{GonzalezSanz.Loubes.Weed.2024.weaklimits,Goldfeld.et.al.2024.EJS}, the conjecture was proven  for smooth cost functions. Finally,  \cite{GonzalezSanz.Hundrieser..2023.Beyond} relaxed the regularity assumption on the cost, proving the whole conjecture of \cite{Harchaoui.Liu.Pal.2020.BJ.}. However, this generalization introduces significant technical challenges, as the linearization of the empirical Schrödinger system  { must now be carried out within} the sequence of Hilbert spaces \( L^2_0(P_n) \), consisting of \( P_n \)-centered elements in \( L^2(P_n) \). A key idea in their approach is to express the derivative of \( \Gamma_n \), which takes the form \( I + \mathcal{A}_n \), through its series expansion
$ \sum_{k=1}^{\infty} (-\mathcal{A}_n)^k.$ 
Since the operator norm of \( \mathcal{A}_n \) is uniformly bounded by a constant, strictly smaller than one, this expansion provides a powerful tool for the analysis. In particular, it enables us to interpret the difference \( (f_\epsilon - f_{\epsilon,n}, g_\epsilon - g_{\epsilon,n}) \)  as an infinite-order \( V \)-statistic, facilitating the derivation of the central limit theorem in this more general setting.  
\begin{theorem}
 Fix $\eta:\R^d\times \R^d\to \R$.    Assume that the cost function $c:\R^d\times \R^d\to \R$ and $\eta$ are  measurable in the support of $P\otimes Q$.  Then it holds that 
    $$ \sqrt{n}\int \eta d(\pi_{\epsilon,n}-\pi_\epsilon)
    { \underset w \to } \mathcal{N}(0,\sigma^2_{\lambda}(\eta)).$$
\end{theorem}
The variance of the limit described in the previous result has a technical expression and can be found in \cite{GonzalezSanz.Hundrieser..2023.Beyond,GonzalezSanz.Loubes.Weed.2024.weaklimits,Harchaoui.Liu.Pal.2020.BJ.}, see also \cite{klatt.et.al.2020.SIMODS,Hundrieser.et.al2021.AoAP} for the discrete setting.  \\
Applications of the previous central limit theorems for the potentials include the derivation of smoothness and  convergence rates of Gaussian processes indexed by distributions using a kernel based on the entropic transportation cost, as defined in \cite{bachoc.2023.Improved,Bachoc.2023.AISTATS}. They also help provide confidence bands for the Sinkhorn cost as initially proposed by \cite{Cuturi2013SinkhornDL}, limits for the entropic optimal transport map \cite{Seguy.et.al.2018.ICLR} and optimal transport-based colocalization curves \cite{klatt.et.al.2020.SIMODS}. {In the next section, we will study the behavior of the plug-in estimator for Sinkhorn divergence. We will see how this can be used for statistical inference.}

\subsubsection{Sinkhorn Divergence}
As noted in the Introduction, the OT cost (for some choices of cost functions) induces a distance over the space of distributions, making it a powerful tool for data analysis and statistics. For instance, goodness-of-fit  and independence tests have been proposed using this property \cite{GonzlezDelgado.et.al.2023.EJS,Hallin.Mordant.Segers.2021.EJS}. However, as we have seen earlier, the curse of dimensionality hinders its use in general dimensions, thus reducing its applicability. The EOT cost with Euclidean squared cost seems a natural substitute but it does not define a distance, nor even a divergence, between probability measures. It does not  satisfy the triangle inequality and the property that it equals zero if and only if the measures are identical. A natural remedy is  proposed by \cite{genevay.Peyre.Cuturi.2018.AISTATs} where they define  the Sinkhorn divergence as follows
$${\rm D}_\epsilon(P,Q) = \mathcal{E}_\epsilon(P,Q) - \frac{1}{2}(\mathcal{E}_\epsilon(P,P) + \mathcal{E}_\epsilon(Q,Q)), $$
with the cost function being  squared Euclidean distance.  
Although the Sinkhorn divergence does not satisfy the triangle inequality, {it  is symmetric, nonnegative  and   equal to zero if and only if \( P = Q \), }  cf.~\cite{Feydy.et.al.2019.PMLR}. This makes it a natural substitute for the Wasserstein distance in high-dimensional data analysis problems. The results obtained by \cite{GonzalezSanz.Loubes.Weed.2024.weaklimits,Goldfeld.et.al.2024.EJS} enable the construction of confidence bands for the Sinkhorn divergence under the two hypotheses \( H_0 : P = Q \) and \( H_1 : P \neq Q \). The test under the alternative hypothesis can be easily derived from the results for the cost. Under the null hypothesis a second-order analysis is required, where the central limit theorem for the potentials plays a key role. Here, we present the result in the form described in \cite{GonzalezSanz.Loubes.Weed.2024.weaklimits}. For a description of the limiting variances we refer to \cite{GonzalezSanz.Hundrieser..2023.Beyond,GonzalezSanz.Loubes.Weed.2024.weaklimits} for the continuous case and \cite{klatt.et.al.2020.SIMODS,Hundrieser.et.al2021.AoAP} for the discrete case.  
   
\begin{theorem}
    Let the cost function be the squared Euclidean distance. Then it holds that  
    \begin{enumerate}
        \item for $P\neq Q$, 
       $$
    \sqrt{n}({\rm D}_\epsilon(P_n,Q_n)- {\rm D}_\epsilon(P,Q)) \underset{w}{\to } \mathcal{N}(0, \sigma^2(P,Q)),
$$
        \item and for $P=Q$, 
        \begin{equation*}
         n\,{\rm D}_\epsilon(P_n,P)\underset{w}{\to } \frac{\epsilon}{2}\sum_{i=1}^{\infty}\lambda_i^2 N_i^2,
   \end{equation*}
where $ \{N_i\}_{i\in \N}$ is a sequence of i.i.d.\ random variables with $N_i\sim  N(0,1)$ and   $\{\lambda_{i}\}_{i\in \N}\subset [0, \infty)$ is such that $\sum_{i=1}^\infty\lambda_{i}^2< \infty. $ 
    \end{enumerate}

\end{theorem}

The limiting distribution depends on $P$, hence, the test statistic ${\rm D}_\epsilon(P_n,P)$ is not distribution-free. Finding  consistent estimators of the sequence $\{\lambda_{i}\}_{i\in \N}\subset [0, \infty)$ is an open problem. We believe that  the series expansion of the linearized operators provided in \cite{GonzalezSanz.Hundrieser..2023.Beyond} might give a consistent estimator of these parameters. 

\subsubsection{Lower complexity adaptation of entropy regularized optimal transport}\label{Section.LCA-EOT}
{
We saw before that the OT problem suffers from the curse of dimensionality, but following the LCA principle, EOT avoids the curse of dimensionality, with the dependence hidden in $C(\epsilon)$ in
\begin{equation}
    \label{eq:Sinkhorn-Bound}
     \E[ |\mathcal{E}_\epsilon(P_n,Q) - \mathcal{E}_\epsilon(P,Q)|] \leq \frac{C(\epsilon)}{n^{\frac{1}{2}}},
\end{equation}
which explodes as $\epsilon$ decreases. The first paper showing this was \cite{genevay.et.al.2019.PMLR}, where the constant $C_\epsilon$ was exponential in $1/\epsilon$. After this,  \cite{Mena.Weed.2019.Nips} showed that  $C_\epsilon \lesssim  \epsilon^{-\frac{d}{2}} $ for smooth costs.   The LCA paradigm in EOT can be formulated in two different ways: (i) does the constant $C_\epsilon$ depend only on the  smallest complexity of the measures $P$ and $Q$? (ii) Can we find a constant $C$, independent of $\epsilon$, such that 
$$ \E[ |\mathcal{E}_\epsilon(P_n,Q) - \mathcal{E}_\epsilon(P,Q)|] \leq \frac{C}{n^{\alpha}} $$
with $\alpha$ depending only on the  smallest complexity of the measures $P$ and $Q$? The two different notions of LCA were addressed in two different parallel works \cite{Stromme-Minim} and \cite{Groppe.Hundrieser.JMLR}. In \cite{Stromme-Minim} it is shown that, for Lipschitz costs:
$$ C_\epsilon \lesssim {\min(\mathcal{N}(\epsilon, {\rm supp}(P),\|\cdot\|),  \mathcal{N}(\epsilon, {\rm supp}(Q), \|\cdot\|)},$$ 
where we recall that  $ \mathcal{N}(\epsilon, \Omega,\|\cdot\|))$ denotes the $\epsilon$-covering number of $\Omega$ with the Euclidean norm.  In the same work, the author also obtained  a quantitative estimate of the EOT potentials and plans for probability measures supported on smooth manifolds without boundary, by showing an empirical PL inequality of the dual functional. The proof relies on a precise control of the eigenvalue gap of a random graph Laplacian, which \cite{Stromme-Minim} obtains from the estimates given in \cite{GarcaTrillos2019}. Remarkably, these estimates depend crucially on the control of the empirical  $\infty$-Wasserstein distance obtained in \cite{GarciaTrillos_Slepeev_2015}.}

{
On the other hand, \cite{Groppe.Hundrieser.JMLR} focused on the second notion of LCA.\footnote{In \cite{Groppe.Hundrieser.JMLR}, the authors show  similar results as  those by \cite{Stromme-Minim} but limited to smooth costs.  } They based their proof on empirical processes theory --- i.e., controlling covering numbers of EOT potentials. They assumed that the support of $P$ (or  of $Q$) is in Lipschitz homeomorphism with a domain of $\R^{m}$ with $m\leq d$.  They showed that, for Lipschitz costs, 
$$  \E[ |\mathcal{E}_\epsilon(P_n,Q) - \mathcal{E}_\epsilon(P,Q)|] \leq C  \begin{cases}
    n^{-\frac{1}{2}}  & \text{if }m<2,\\
     n^{-\frac{1}{2}} \log(n)  & \text{if }m =2, \\
     n^{-\frac{1}{m}}  & \text{otherwise},
\end{cases} $$
where the constant $C$ is independent of $\epsilon$. The rates can be improved for $C^2$ costs, see Theorem~15, ibid. 
}
\subsection{Smooth optimal transport}
{The curse of dimensionality is not
specific to classical OT and is inherent to other statistical problems.} One of the most classical examples is the density estimation problem, where the kernel density estimator 
\begin{equation}
    \label{eq:Kernel-density-est}
    \hat{p}_h^{(k)}(x)= \frac{1}{h^d n} \sum_{i=1}^n K\left(\frac{x-X_i}{h}\right)
\end{equation}
or the wavelet  density estimator $\hat{p}_h^{(w)}$ have rates of convergence depending on two factors: the smoothness of the density $p$ of $P$ and the dimension. This means that some plug-in kernel/wavelet density estimators adapt to the smoothness of the unknown density, giving almost parametric rates if the dimension is small compared to the degree of smoothness. 

As we have seen, the plug-in estimator of the OT cost adapts to the underlying complexity of the measures  \cite{Hundrieser.et.al.2024.AIHP} but not to the smoothness of the densities \cite{Fournier.Guillin.2014.PTRF}. Hence, different types of estimators are required to leverage the a priori smoothness information of the density. The first contribution in this direction was due to \cite{Hutter.Rigollet.2021.AOS}, where the minimax rates of convergence for the estimation of smooth optimal transport maps are provided. The upper bound is found via an estimator based on the minimization of the empirical semi-dual OT problem, over truncated
wavelet expansions, which turns out to be computationally unfeasible. More tractable estimators have been proposed in \cite{Deb.Ghosal.Sen.2021.Nips,Gunsilius.2021.Econometric,Manole.et.al.2024.AoS} based on solving the OT between the continuous smooth density estimators described above. Obviously, to approximate the OT related quantities the density estimators should converge to their population counterparts. Hence, bandwidths (for kernel estimators) or truncations (for wavelet) should be adapted to the sample size. 
 
\subsubsection{Smooth Wasserstein distances}
In this section we focus on kernel-type  estimators where the bandwidth does not change with the sample size. In particular, we focus on  
 the Smooth $p$-Wasserstein distance, proposed by \cite{Goldfeld.et.al.2020.IEE.TrIinfTh,goldfeld.Kristjan.2020.AIstats},   
$$ \mathcal{W}_p^\sigma(P,Q)=\mathcal{W}_p(\mathcal{N}(0,\sigma^2{{\rm I}_d} )*P, \mathcal{N}(0,\sigma^2{{\rm I}_d} )*Q), $$
where 
$ \mathcal{N}(0,\sigma^2{{\rm I}_d} )*P$ denotes the convolution of $P$ with a $\mathcal{N}(0,\sigma^2{{\rm I}_d} )$ measure. The smooth OT-plan is denoted by $\pi_\sigma$, which by \cite{Gangbo.McCann.1996.ActaMath} is concentrated on the graph of a map (for $p\neq 1$).

A little thought shows that $\mathcal{W}_p^\sigma$ defines a distance in the space of probability measures with finite $p$-th moment.
The smooth OT plans and costs converge to their unregularized counterparts as the regularization parameter decreases. Hence, its empirical version $\mathcal{W}_p^\sigma(P_n,Q)$, if computable, can be used as a one-sample test statistic.

As EOT, for a fixed regularization parameter,   smooth OT avoids the curse of dimensionality and the following limiting theorem  hold{s} \cite{Goldfeld.et.al.2024.AoAP}. We state the result for compactly supported measures, however it holds in greater generality, cf.~Equation~(4) in \cite{Goldfeld.et.al.2024.AoAP}. 
\begin{theorem}
Fix $\sigma>0$ and $p>1$. Assume that $P$ and $Q$ are compactly supported. Then  
    \begin{enumerate}
        \item[(i)] If $P=Q$, $\sqrt{n}\mathcal{W}_p^\sigma (P_n,P)$  converges in distribution to the dual norm of a  centered Gaussian process in a negative Sobolev space,
     \item[(ii)] If $P\neq Q$ $$\sqrt{n}(\mathcal{W}_p^\sigma (P_n,Q)- \mathcal{W}_p^\sigma (P,Q))$$ converges in distribution  to a {non-degenerate} centered Gaussian r.v.  
    \end{enumerate}
\end{theorem}
\begin{remark}
    A similar result holds for $p=1$,  where under the null hypothesis $P=Q$,  the limit 
$\sqrt{n}\mathcal{W}_p^\sigma (P_n,P)$ also has   as limiting distribution a dual norm \cite{sadhu.Goldfeld.Kato.2022.Preprint,Sloan.Goldfeld.Kato.2021.PMLR}. However, for $p=1$ and  under the alternative, the limiting distribution is not in general Gaussian due to the non uniqueness of the potentials.
\end{remark}
 
The proof technique follows the functional delta-method and the fact that first variations of the Wasserstein distances are dual Sobolev norms,  cf.~\cite[Exercise 22.20]{villani} and \cite[Remark~3.2]{Goldfeld.et.al.2024.AoAP}. The authors also show the consistency of the bootstrap (cf.~Proposition~3, ibidem).  The same proofs adapt for variations of smooth Wasserstein distances,  based on  convolutions with different smooth measures such as the Neumann heat semigroup  of a bounded strictly convex set $\Omega$ with smooth boundary.  Assuming that $P$ and $Q$ are supported in $\Omega$, the linearization of the optimal transport map for the squared-Euclidean cost is possible following the approach described in the next section. Hence, in this setting central limit theorems for the regularized maps hold for sufficiently smooth domains. 

Finally, we point out that the contraction of the heat semigroup w.r.t.\ the Wasserstein distance, i.e., the inequality 
$$ \mathcal{W}_p^\sigma (P_n,Q)\leq e^{-c\,\sigma } \mathcal{W}_p (P_n,Q), $$
implies that 
the one-sample goodness-of-fit test statistic $\mathcal{W}_p^\sigma (P_n,Q)$ provides uniform conservative control of $\mathcal{W}_p (P_n,Q)$ under the alternative $H_1:Q\neq P$. However, the same bound implies that the test will be less powerful as $\sigma$ increases. This is something expected, as $ \mathcal{N}(0,\sigma^2{{\rm I}_d} )*P$ becomes exponentially close to the fundamental  solution of the heat equation at time $t=\sigma$ as the time parameter increases. Hence both $ \mathcal{N}(0,\sigma^2{{\rm I}_d} )*P_n$ and $ \mathcal{N}(0,\sigma^2{{\rm I}_d} )*Q$ become almost indistinguishable for large values of $\sigma$. 
 
\subsubsection{ Estimation of  smooth optimal transport maps} 
We saw that EOT and  Smooth OT avoid the curse of dimensionality, with plans and cost satisfying   CLTs. However, the centering variable differs from the unregularized counterpart. In this section, we focus on a class of estimators that approximate sufficiently smooth OT maps satisfying a CLT. One notable example is the work of \cite{Manole.et.al.2024.Preprint}, who proved a CLT for kernel density–based methods when the measures are supported on the flat torus (i.e., periodic measures with the squared periodic distance cost). Let us review  the main ideas of \cite{Manole.et.al.2024.Preprint} and the difficulties for its Euclidean extension.

Let $\hat{P}_h^{(k)}$ be the probability measure with density \eqref{eq:Kernel-density-est}. Let $\hat{\Omega}_{P}$, $\Omega_P$ and $\Omega_Q$ be the supports of $\hat{P}_h^{(k)}$, $P$ and $Q$, respectively. Assume that there exist smooth convex potentials $\hat{\varphi}_{h}$ and $\varphi$ whose gradients $\nabla \hat{\varphi}_{h}$ and $\nabla \varphi$ push forward $Q$ to $\hat{P}_h^{(k)}$ and $P$, respectively.  Then they solve Monge-Ampère equations 
 $$ \det(\nabla^2\hat{\varphi}_{h} ) = \frac{q}{\hat{p}_h^{(k)}\circ \nabla\hat{\varphi}_{h}} \quad {\rm s.t.}\quad \nabla\hat{\varphi}_{h}(\Omega_Q)\subset \hat{\Omega}_{P} $$
and 
 $$ \det(\nabla^2{\varphi} ) = \frac{q}{p\circ \nabla {\varphi}} \quad {\rm s.t.}\quad \nabla {\varphi}(\Omega_Q)\subset {\Omega_P}. $$
In the flat torus, the boundary conditions are periodic. 
The key idea of \cite{Manole.et.al.2024.Preprint} is to treat these equations as a Z-estimation problem. Hence, one needs to show the Fréchet differentiability of the Monge-Ampère equation at the pair $(P, \varphi) $ with  partial derivative w.r.t.\ $\varphi$ admitting  bounded inverse in some appropriate Banach spaces of differentiable functions.   The interior part of the equation can be easily linearized if the densities are bounded away from zero and belong to $\mathcal{C}^{1,\alpha}$ (cf.~\cite[Chapter~3]{Figalli.2017.Book}). The resulting linearized operator is elliptic. Under the additional assumption 
\begin{equation}
    \label{Assumption}
    \lambda I\leq\nabla^2 \varphi  \leq \Lambda I , \quad 0<\lambda\leq \Lambda<\infty,
\end{equation}
the linearized equation is non-degenerate elliptic.

In the flat torus, this linearized equation admits a unique classical solution (up to additive constants) for smooth data ---hence, the linearized operator is invertible---, see \cite{Manole.et.al.2024.Preprint}. In the Euclidean case, the boundary conditions need to be linearized too. This is the first {major} difficulty for Euclidean data. However, this one can be solved if the smoothed density $\hat{P}_h^{(k)}$ has a support $\hat{\Omega}_{P} $ converging to ${\Omega_P} $  in a suitable smooth way. For instance, if $\hat{\Omega}_{P}= {\Omega_P} $ holds for all $n\in \N$. In this case, one obtains that the
linearized equation is elliptic with oblique boundary conditions and its solvability holds as in the periodic case (cf.~\cite{GonzalezSanz.Sheng.2024,loeper2005regularity}).  One thus {has} the relation 
$$ \nabla\hat{\varphi}_{h}-\nabla\varphi= \mathbb{L}^{-1} (\hat{p}_h^{(k)}-p) +o\left(\|\hat{p}_h^{(k)}-p\|_{\mathcal{C}^{1,\alpha}} \right),$$
where $\mathbb{L}^{-1}$ is the inverse of the linearized Monge-Ampère equation at  $\varphi $.  The error  term in the previous display is not sharp and it  can be improved to $ O\left(\|\hat{p}_h^{(k)}-p\|_{\mathcal{C}^{0,\alpha}} \|\hat{p}_h^{(k)}-p\|_{\mathcal{C}^{1,\alpha}} \right)$. The suboptimal error derived in the first argument is due to the fact that the linearization of the right-hand side of the Monge-Ampère equation requires control of a further derivative of the difference $\hat{p}_h^{(k)}-p$.     This can be improved  by instead linearizing  the {right-hand} side of 
\begin{equation}
    \label{MA:modified}
   ({p}\circ \nabla\hat{\varphi}_{h}) \cdot \det(\nabla^2\hat{\varphi}_{h} ) = \frac{   ({p}\circ \nabla\hat{\varphi}_{h})  \cdot q}{\hat{p}_h^{(k)}\circ \nabla\hat{\varphi}_{h}} ,
\end{equation}
with the same boundary conditions (oblique or periodic, depending on the setting). The linearization of the left-hand side of \eqref{MA:modified} is again a non-degenerate elliptic operator, while the linearization of the right-hand side of \eqref{MA:modified} behaves in $\mathcal{C}^{0,\alpha}$ as 
$$ \left\| \frac{   ({p}\circ \nabla\hat{\varphi}_{h})  \cdot q}{\hat{p}_h^{(k)}\circ \nabla\hat{\varphi}_{h}}-q\right\|_{\mathcal{C}^{0,\alpha}}\leq C \|\hat{p}_h^{(k)}-p\|_{\mathcal{C}^{0,\alpha}}. $$
From here we derive the estimate
\begin{align}\label{eq:}
    \| \nabla\hat{\varphi}_{h}-\nabla\varphi \|_{\mathcal{C}^{1,\alpha}} \leq C \|\hat{p}_h^{(k)}-p\|_{\mathcal{C}^{0,\alpha}},
\end{align}
which allows us to obtain the claimed improvement (cf.~\cite[Theorem~2]{Manole.et.al.2024.Preprint}). Hence, if there were a central limit theorem for $\hat{p}_h^{(k)}$ as an element of $\mathcal{C}^{0,\alpha}$ one would get a central limit theorem for $\nabla\hat{\varphi}_{h}$ in $\mathcal{C}^{1,\alpha}$. This however, cannot hold (cf.\ \cite[Theorem~7]{Manole.et.al.2024.Preprint}). Limits can only be obtained in negative Sobolev norms or pointwise for dimensions greater than $3$.

To derive {pointwise limits},  \cite{Manole.et.al.2024.Preprint} developed a bias-variance decomposition of $\nabla\hat{\varphi}_{h}-\nabla\varphi$, where the bias term 
is represented by $\nabla{\varphi}_{h}-\nabla\varphi$ and the variance by $\nabla\hat{\varphi}_{h}-\nabla\varphi_h$, for $\nabla{\varphi}_{h}$ denoting  the transport map from $Q$ to the probability measure $P$ convoluted with the kernel with bandwidth $h$, call its density $p_h$.  Then, using the estimates 
$$ \| p_h-p\|_{\mathcal{C}^{0,\alpha}} \leq h^{s-\alpha}, \quad  \| p_h-\hat{p}_h\|_{\mathcal{C}^{0,\alpha}} \leq \sqrt{\frac{\log(h^{-1})}{n\,h^{2\alpha+d}}}, $$
where $s$ is the number of continuous derivatives of the densities, and the previous estimates,
\cite{Manole.et.al.2024.Preprint} obtains the bound 
$$ \|\nabla\hat{\varphi}_{h}-\nabla\varphi -\mathbb{L}^{-1} (\hat{p}_h^{(k)}-p)\|_{\mathcal{C}^{2,\alpha}}\leq h^{2s-1-3\alpha} + \sqrt{\frac{1}{n h^{d+2+4\alpha}}}, $$
for appropriate choices of the bandwidth $h$. 
The next step is to show that $\mathbb{L}^{-1} (\hat{p}_h^{(k)}-p)$ has a non-degenerate point-wise limit with rate slower than the right-hand side of   the previous display. The first step is to show that the bias term vanishes faster than the variance term. Then, by means of a Fourier analysis to guarantee  Lyapunov’s condition, {the central limit theorem for the variance term can then be established.} This last step and the control of the boundary bias of the kernel density estimators are the main hurdles for the Euclidean generalizations of the following  result derived by \cite{Manole.et.al.2024.Preprint} for periodic data. 

\begin{theorem}
    Fix $d\geq 3$ and { $s> 2$}. Let $p,q$ be  densities in $\mathbb{R}^d/\mathbb{Z}^d$ such that $\log(p)$ and $\log(q)$ are $\mathcal{C}^s$. Assume that $K\in \mathcal{C}^\infty_c((0,1)^d)$ is an even kernel such that its Fourier transform $\mathcal{F}(K)$ satisfies
    $$\sup_{\|x\|\neq 0} \frac{|[\mathcal{F}(K)](x)-1|}{\|x\|^{s+1}}<\infty. $$
  Then for any $x\in \mathbb{R}^d/\mathbb{Z}^d$ and $h=c\,n^{-\beta}$ for some $c>0$ and $\frac{1}{d+4}<\beta <\frac{1}{d+2s}$, the sequence $$\sqrt{n h^{d-2}}\left(\nabla\hat{\varphi}_{h}(x)-\nabla\varphi(x)\right)$$ 
    admits a non-degenerate Gaussian limit in distribution. 
\end{theorem}

\section{Further OT-related limits}\label{Section:Further-OT}
\subsection{Sliced optimal transport}\label{sect:sliced}

A different, alternative way to mitigate the computational and statistical challenges that arise in high-dimensional OT problems is to try to leverage the properties of the one-dimensional setting by computing OT distances between projected distributions. Among the various notions explored in the existing literature, we focus on the most prominent ones: the sliced and max-sliced Wasserstein distances, defined respectively as
\begin{align}\label{eq:sliced}
&\mathcal{S}_p(P_n,Q)= \int_{\mathbb S^{d-1}} \mathcal{T}_p \left(\operatorname{Pr}_u \sharp P_n,\operatorname{Pr}_u \sharp  Q\right)\  d\sigma(u) \ , \\
\label{eq:max-sliced}&{\mathcal{M}_p} (P_n,Q)= \sup_{u\in\mathbb S^{d-1}} \mathcal{T}_p \left(\operatorname{Pr}_u \sharp P_n,\operatorname{Pr}_u \sharp  Q\right) \ ,
\end{align}
where $\sigma$ denotes the uniform measure on the unit sphere $\mathbb S^{d-1}$ and $\operatorname{Pr}_u$ the projection onto $u\in\mathbb S^{d-1}$. 
For the max-sliced version (and a subspace generalization)
\cite{niles2022estimation} obtained expectation bounds, with rates independent of the dimension. \cite{Manole.et.al.EJS} adapts the previous bounds to the sliced setting, extending their validity to a trimmed version of \eqref{eq:sliced}, and proving the first distributional limit for the sliced Wasserstein distance 
\begin{equation}\label{biased_CLT_sliced}
\sqrt{n}(\mathcal{S}_p(P_n,Q)-\mathcal{S}_p(P,Q))\underset w \to N(0,v_p^2(P,Q)),
\end{equation}
and its extension for the trimmed version, {for $p > 1$ and $P \neq Q$ (in the sense
that the asymptotic distribution vanishes if $P = Q$, as in Theorem \ref{1dCLT})}. 
Denoting by $F_{u,n}$ and $F_u$ the CDF of the projected distributions along the direction of $u$, their asymptotic result is based on the weak convergence of the empirical process $\mathbb G_n(u,x) = \sqrt{n}(F_{u,n}(x)-F_{u}(x)) $, along with Hadamard differentiability and the functional delta method. Since this work primarily focuses on the trimmed version, the assumptions are rather strong for the untrimmed setting, involving finiteness of the functional 
\begin{equation}
    SJ_{\infty}(P) = \underset{0<t<1}{\operatorname{essup}} \   \frac{1}{f_u(F_u^{-1}(t))} <\infty\ \ ,
\end{equation}
where $f_u$ denotes the density associated to $F_u$.
An important step in the derivation of distributional limits for projection-based distances was the unifying approach presented in \cite{Xi.Weed.2022.Nips}. Using similar ideas to those in \cite{Manole.et.al.EJS}, and inspired by the duality-based approach in \cite{Hundrieser.et.al.BJ.2024.Unifying}, \cite{Xi.Weed.2022.Nips} investigates the weak convergence of the sliced process $\mathbb{G}_n(u) = \sqrt{n}(\mathcal{T}_p(P_n^u, Q^u) - \mathcal{T}_p(P^u, Q^u))$,
under the assumptions of compactly supported probability distributions such that the support of the projected probabilities is an interval. This ensures the uniqueness (up to constants) of the optimal transport (OT) potentials for each projection. Weak convergence for various projection-based distance notions follows directly from Hadamard differentiability and the functional delta method. Specifically, both sliced and max-sliced distances are considered. The Hadamard differential for the sliced version is linear, which implies the centered Gaussian limit \eqref{biased_CLT_sliced}, but this is not necessarily the case for the max-sliced version. {The main results primarily concern the case $p > 1$, $P \neq Q$, although some insights are also provided
for $p = 1$.}

Independently, \cite{Goldfeld.et.al.2022.statisticalinferenceregularizedoptimal} derived distributional limits for the sliced and max-sliced Wasserstein distances as a consequence of techniques more closely related to optimal transport. In particular, they leverage the dual expression of the one-dimensional Wasserstein distance with $p=1$ as a supremum over Lipschitz functions to demonstrate weak convergence of the empirical process indexed by the functions $\varphi \circ \operatorname{Pr}_u$, where $\varphi$ is 1-Lipschitz and $u \in \mathbb{S}^{d-1}$. Then, they conclude weak convergence from the extended functional delta method under the assumption of compactly supported probabilities with convex support for $p > 1$ {and $P\neq Q$}, and under mild moment assumptions for $p = 1$. The assumptions required for $p = 1$ were further refined by \cite{xu2022central}. 

Later, \cite{Hundrieser.et.al.2024.SPA.Estimated} provided a refined version of the results in \cite{Xi.Weed.2022.Nips} and \cite{Goldfeld.et.al.2022.statisticalinferenceregularizedoptimal}, again for compactly supported probabilities but with slightly weaker assumptions on the supports, using similar arguments to those in \cite{Xi.Weed.2022.Nips}. In a different vein, employing slightly different proof techniques, this work also explores, with great generality, distributional limits for empirical transport problems where the cost function is also estimated from the data.

In contrast to these unifying approaches, two recent papers have again adopted problem-specific strategies to improve existing results. \cite{han2024max} proved completely dimension-free expectation bounds for $\mathcal M_p$, which can be extended to infinite-dimensional Hilbert spaces. The recent paper \cite{sliced_arxiv} gives a new CLT for $\mathcal S_p$ based on the Efron-Stein linearization approach presented in Section \ref{general_dimension} in this paper. We quote here a version of the CLT for the fluctuation.

\begin{theorem}\label{teor:CLT_sliced_fluctutation}
Assume $p>1$, $\delta>0$. Let  $P$ and $Q$ be probabilities on $\mathbb{R}^d$ with finite moments of order $2p+\delta$. Assume further {that} $P$ is absolutely continuous with negligible boundary, and $\operatorname{int}({\rm supp}({P}))$ is connected. Then 
\begin{equation}\label{unbiased_CLT_sliced}
\sqrt{n}\left(\mathcal{S}_p(P_n,Q)-\mathbb E[\mathcal{S}_p(P_n,Q)]\right)\underset w \to N(0,v_p^2(P,Q)),
\end{equation}
for some $v_p^2(P,Q)\geq 0$ {and strictly positive unless $P=Q$.}
\end{theorem}
We refer to \cite{sliced_arxiv} for a precise description of the limiting variance and a proof.

Unlike standard OT, $\mathcal S_p$ retains the favorable properties of the one-dimensional setting and it is possible to replace {the centering constant in  \eqref{unbiased_CLT_sliced} by the population counterpart}. {Theorem 4.2 in \cite{sliced_arxiv}} provides sufficient conditions for this goal. The assumptions on the density in Theorem \ref{teor:bias_1d} must be imposed for all the projected densities  $f_u$, and $J_\alpha$ must be replaced with its integrated version, $SJ_\alpha(P)= \int_{\mathbb S^{d-1}}J_\alpha(\operatorname{Pr}_u\sharp P)d\sigma(u)$, to ensure 
\begin{equation}\label{bias_convergence_sliced}
\sqrt{n}(\mathbb E[\mathcal{S}_p(P_n,Q)]-\mathcal{S}_p(P,Q)) \to 0  \ .
\end{equation}
Combining \eqref{unbiased_CLT_sliced} and \eqref{bias_convergence_sliced}, \cite{sliced_arxiv} provides the first CLT for $S_p$, {$p>1$}, which is also valid for measures without compact support.
\subsection{Gromov--Wasserstein distance}
{
Let $(\mathcal{X}, d_X, P)$ and $(\mathcal{Y}, d_Y, Q)$ be two metric measure spaces.
For $p \geq 1$, the $(p,q)$--Gromov--Wasserstein (GW) distance of order $p$ between them is defined as
\begin{equation}    \label{eq:Grom-Wass}
     \mathcal{GW}_{p,q}(P, Q) =\Bigg( \inf_{\pi \in \Pi(P, Q)} \int \int |d_X^q(x,x') - d_Y^q(y,y')|^p  
    d(\pi \otimes \pi)(x, y, x', y') \Bigg)^{\frac{1}{p}}.
\end{equation}
GW is a distance between metric probability spaces modulo isomorphisms.\footnote{In the category of metric probability spaces, two probability metric spaces $(\mathcal{X}, d_X, P)$ and $(\mathcal{Y}, d_Y, Q)$ are isomorphic if there exists an isometry between $(\mathcal{X}, d_X)$ and $(\mathcal{Y}, d_Y)$ pushing $P$ forward to $Q$.} GW was first proposed by \cite{Memoli-GW.FoCM.2011}. Some applications of GW can be found in \cite{Demetci.et.al.J.comp.bio.2022}, \cite{xu2019gromov} and \cite{Han.Rigolet.Stepaniants.SIMODs.2025}. For the computational aspects we refer to \cite{Rioux.Godfeld.Kato.GW.JMLR.2024}. The geometry and geodesic gradient flows of GW have been studied in \cite{zhang2025gradient}. The paper that provides the statistical complexity of GW is \cite{Zhang.Goldfeld.Mroueh.AoS.Gromov}, where further results such as duality and convergence rates of the entropic regularization are also shown.  The main result of \cite{Zhang.Goldfeld.Mroueh.AoS.Gromov} can be summarized as follows. 
\begin{theorem}\label{Theorem:Gromov-Wasserstein}
    Let $P\in \mathcal{P}(\R^{d_1})$ and $Q\in \mathcal{P}(\R^{d_2})$ be compactly supported. Fix $d=\min(d_1,d_2)$. Then, 
$$\E[ |\mathcal{GW}_{2,2}^2(P, Q) -  \mathcal{GW}_{2,2}^2(P_n, Q)|] 
         \lesssim  \frac{\log(n){\bf 1}_{d=4}}{n^{\frac{2}{\max(d,4)}}},  $$
as $n\to \infty$. 
\end{theorem}
The bound of Theorem~\ref{Theorem:Gromov-Wasserstein} is sharp, up to poly-log factors. Hence, GW suffers from the curse of dimensionality in a very similar way as the original OT problem. The proof uses the representation (see Corollary~4.1 in \cite{Zhang.Goldfeld.Mroueh.AoS.Gromov})} 
{
$$ \mathcal{GW}_{2,2}(P,Q) = \inf_{A\in \R^{d_X\times d_Y}} 32 \| A\|_{{\rm Fr}}^2+ \mathcal{T}_{{c}_A} (P,Q),$$
where $ \| A\|_{{\rm Fr}}$ denotes the Frobenius norm and $${c}_A(x,y)=-4\|x\|^2\|y\|^2-32 \langle x, A y\rangle .$$
Then GW admits a dual formulation and   $|\mathcal{GW}_{2,2}(P,Q)-\mathcal{GW}_{2,2}(P_n,Q)|$  can be bounded by a supremum of the  empirical processes indexed
by dual OT potentials.  Furthermore, the LCA holds. 
} 

{The same paper studies the entropic version of GW, showing that the rate of convergence of the cost is parametric. To the best of our knowledge, the sample complexity of the plans and dual potentials is unknown for GW.   }

\section{Applications}\label{section:applications}
\subsection{Applications of CLT for OT costs}
Understanding the asymptotic behavior of OT costs enables us to build tests to assess the similarity between distributions. Contrary to usual goodness of fit tests based on Kullback-Leibler distance, tests based on Monge-Kantorovich, a.k.a.\ Wasserstein distance, are better able to capture the structural variability of the distribution of the observations. The first tests were developed for medical or biological applications to capture the inter-individual variability of this type of data where the response is highly influenced by the individual characteristics. When the response was a function, tests were based on registration methods and the natural extension to distributional response was obtained by considering similarity tests based on OT-based distances as in \cite{dupuy2011non} or \cite{freitag2005hadamard} and \cite{freitag2007nonparametric}. In this context, optimal transport is viewed as a way to provide a sound geometrical interpretation of distributions and testing in the Wasserstein space of distributions enables us to obtain a better interpretation of the notion of similarity. Tests in the Wasserstein space are natural  for deformation models \cite{del2019centrala} or for data analysis on distributions as in \cite{cazelles2018geodesic}
 or \cite{boissard2015distribution}. \\
 The problem of assessing bias in algorithmic decisions is similar, to some extent, to previous applications. Actually, bias in AI arises when a variable, which characterizes a group of individuals, affects systematically the behavior of an algorithm. This implies that the decisions or the performance of the algorithm may be different for different subgroups, leading to possible infringement of their fundamental rights and thus to discrimination. Detecting such disparate behavior of the algorithm provides a natural framework for optimal-transport-based goodness-of-fit. For instance in the case where the population is split into a minority and a majority according to the value of a so-called sensitive variable ($A\in \{0,1\}$), testing for bias amounts to testing whether the algorithm or its loss exhibits different behavior for the two groups and looking at the distance between the corresponding conditional distribution when $A=0$ or $A=1$, namely $\mu_{A=0}(f)$ and $\mu_{A=1}(f)$. Note that choosing  Wasserstein distance to evaluate fairness is a relevant choice since OT transport cost is deeply related to bias discovery and bias mitigation as pointed out in \cite{gouic2020projection} and \cite{chzhen2020fair}. Regulations, such as the European AI Act, have determined a threshold $\Delta$ above which the dissimilarity is not acceptable leading to the nonconformity of the AI system. Hence, the corresponding statistical test for fairness using OT asymptotic behavior can be formulated by choosing the null hypothesis as    $$ H_0: \mathcal{W}_p(\mu_{A=1}(f),\mu_{A=0}(f))  \geq \Delta.$$ As pointed out in~\cite{CLTonedimensional}, rejection of the null hypothesis would yield statistical evidence that the algorithm has the same {behavior} over both groups and thus  is compliant with respect to the considered regulation. Fairness {tests based on Wasserstein distance are also provided}  in \cite{10.1145/3442188.3445927} and \cite{si2021testing}.
\subsection{Application to transport-based quantiles}\label{Section:quantiles}

{
Quantiles and ranks are the fundamental tools for semiparametric and nonparametric inference on the real line. However, their definitions are inherently linked to the distribution function, whose definition relies on the `canonical' order of the real line, which does not exist in higher dimensions. Several attempts have been made to generalize the quantile function to the multivariate setting. For instance, by means of depth theory and associated ranks \cite{Zuo-Serfling.AoS}. However, none of the approaches allows for nonparametric distribution-free inference. In general, depth functions do not characterize the distributions (e.g.~\cite{Nagy}), and their associated ranks lack distribution-freeness (see \cite{Hallinetal2021,Konen-Hallin,Hallin-review}).} 

{
Hallin et al.~\cite{Chernozhukovetal2017,Hallinetal2021} defined the transport-based (or center-outward) quantile function $Q_{\pm,P}$ of $P$ as the unique gradient of a convex function pushing a reference measure $\mu$ to $P$. (Such a function exists and is unique thanks to McCann's theorem \cite{McCann}.) The transport-based distribution $F_{\pm,P}$ of $P$ is the inverse of $Q_{\pm,P}$. To have isometric equivariance, the reference measure should be spherically symmetric --- even though other choices have been used in the literature, see \cite{Deb.Ghosal.Sen.2021.Nips,Deb-Sen.JASA}. The sample version ${\widehat F}_{\pm,P}$ of $F_{\pm,P}$ is the empirical transport map (for the Euclidean square cost) from $\mu_n$ --- a counting measure with $n$ atoms such that $\mu_n \xrightarrow{w} \mu$ --- to the empirical measure $P_n$. The associated transport-based ranks are $R_i = {\widehat F}_{\pm,P}(X_i)$ for $i=1, \dots, n$. It is easy to show that the transport-based ranks are uniformly distributed over the permutations of the $n$ atoms of $\mu_n$ and are maximally ancillary in the sense of Basu \cite{Basu}. This allows for nonparametric multivariate rank-based distribution-free testing \cite{Deb-Sen.JASA,ShiDrtonHallinHan2025,Shi-etal-AOS}. The transport-based distribution allows defining quantile regions $\mathcal{R}_\alpha$ with the property that 
$P(\mathcal{R}_\alpha) = \alpha$ for all $\alpha \in (0,1)$. This has been proved to be useful in multiple-output quantile regression \cite{delBarrion-regression} and autoregression \cite{gonzalezsanz.2025nonparametricvectorquantileautoregression,HallinLiu2023,HallinLaVecchiaLiu2022,HallinLaVecchiaLiu2023}. Furthermore, these quantile regions are robust in the sense of the breakdown point \cite{avellamedina.2024breakdownpointtransportbasedquantiles,paindaveine.2024.robustnesssemidiscreteoptimaltransport}.} 

{
However, the curse of dimensionality arises in the estimation of transport-based quantiles. The results summarized in this paper allow for finding the convergence rates of the transport-based quantiles and their regularized approximations. }

{
It is still an open problem to find the minimax rates for estimating the conditional OT map, which would have applications in regression \cite{delBarrion-regression}. A recent work provides non-sharp rates for the autoregression transport-based quantiles under strong mixing conditions \cite{gonzalezsanz.2025nonparametricvectorquantileautoregression}; the proof techniques seem adaptable to the regression setting. }

\section{Open problems}
While the distributional limit theory for OT is well developed by now and we have presented what we believe to be a reasonably complete description of it, some aspects of the theory are not completely understood. We would like to end this paper with a small sample of  open problems that we believe deserve further investigation.

The fluctuation CLT for OT (Theorem \ref{general1dCLT}) is valid in general dimension for every $c_p$ cost with $p>1$. Given the minimal assumptions needed for the case $p=1$ in the one-dimensional case one could wonder whether a similar result holds in higher dimension. As we have seen, $\sqrt{n}(\mathcal{T}_1(P_n,Q)-\mathbb{E}[\mathcal{T}_1(P_n,Q)])$ is {stochastically} bounded in any dimension. We think it would be of interest to get an answer to the following:
\begin{Problem} Can we find conditions when $d>1$ under which 
$$\sqrt{n}(\mathcal{T}_1(P_n,Q)-\mathbb{E}[\mathcal{T}_1(P_n,Q)])$$
converges weakly? Is it possible to get a Gaussian limit? 
\end{Problem} 

Even in the one-dimensional setup some questions remain unsolved. As an example, we can mention the weak convergence of the quantile process discussed in Theorem \ref{CLT_quantile_process}. The assumption $J_p(P)<\infty$ is a necessary condition in the result, since it is necessary for the integrability of the weighted Brownian bridge
in the limit. However, the monotonicity of $f\circ F^{-1}$ might not be necessary. It is not in the case $p=1$, as shown in the recent paper \cite{beare2025necessarysufficientconditionsconvergence}. With this motivation we think it is worth considering the next question.
\begin{Problem} Can we find necessary and sufficient conditions for the weak convergence of the quantile process
$$v_n(t)=\sqrt{n}(F_n^{-1}(t)-F^{-1}(t)),\quad 0<t<1,$$
as a random element in $L_p(0,1)$?
\end{Problem} 

In the null case $P=Q$ and strictly convex cost ($p>1$) the fluctuation CLT yields a null limiting distribution. With a faster rate there is still some potential room for a different CLT. Building upon earlier work by Ambrosio and coauthors (see \cite{AmbrosioStraTrevisan}), Michel Ledoux conjectured in \cite{Ledoux_conjecture} that when $P$ is the uniform distribution on the unit square $[0,1]^2$, then
\begin{equation}\label{ledoux}
n(\mathcal{T}_2(P_n,P)-\mathbb{E}[\mathcal{T}_2(P_n,P)])\underset w \to \xi,
\end{equation}
$\xi$ being some centered random variable. A result like \eqref{ledoux} holds under some conditions in dimension one. More generally, {we find the following to be a very interesting (and challenging) question. }

\begin{Problem} 
Does \eqref{ledoux} hold? In general dimension? Can we get a good description of the limit?
\end{Problem} 

Finally, we turn to an apparently simpler, yet open problem.
It is well-known that the optimal transport problem with Euclidean cost might admit several solutions. For instance, if $\mu={\rm Uniform}[0,1]^2$ and $\nu ={\rm Uniform}([2,3]\times [0,1])$, any coupling $\pi$ such that the vertical coordinates remain unchanged is optimal, i.e.,  $ \int | x_2-y_2|d\pi(x,y)=0 $, where $x=(x_1,x_2)$ and $y=(y_1,y_2)$. Note that in this case, since $\pi(y_1>x_1)=1$, we get   
\begin{eqnarray*}
\int \| x-y\|d\pi(x,y)&=& \int |x_1 -y_1|d\pi(x,y)\\
&=& \int (y_1-x_1)d\pi(x,y) =2. 
\end{eqnarray*}
However, it is easy to check that the empirical optimal transport plan is a mapping with probability one. The question is the following.  
\begin{Problem}[\cite{santambrogio}]
    Let $X_1, \dots, X_n\overset{iid}{\sim} {\rm Uniform}[0,1]^2$ and $Y_1, \dots, Y_n\overset{iid}{\sim} {\rm Uniform}([2,3]\times [0,1])$.  Let $\pi_n$ be the empirical  optimal transport plan. Find the limit as $n\to \infty$ of $\pi_n$. 
\end{Problem}

\bibliographystyle{amsplain}
\bibliography{referencias}

\clearpage
\begin{bibunit}[amsplain]

\section*{Supplementary Material} 

This document contains the proofs of several results in \cite{delbarrio2025distributionallimittheoryoptimal}. The notation is the same as in the cited reference. Citations here refer to the references list at the end of this document.

\begin{proof}[Proof of Theorem \ref{1dCLT}] 
The convergence statement \eqref{unbiased_CLT_1d} is Theorem 2.1 in \cite{CLTonedimensional}. The limiting variance is defined as follows. We set $h_p(x)=|x|^p$. Then, $h'_p(x)=\text{sgn}(x) p|x|^{p-1}$ and we set
$c_p(t;F,G)=\int_{F^{-1}(\frac 1 2)}^{F^{-1}(t)} h'_p(s-G^{-1}(F(s))ds$, $0<t<1$ and $\bar{c}_p(t;F,G)=
{c}_p(t;F,G)-\int_0^1 {c}_p(s;F,G)ds$. It can be shown that $c_p(\cdot;F,G)\in L_2(0,1)$ (Lemma A.1 in \cite{CLTonedimensional}). The limiting variance in \eqref{unbiased_CLT_1d}
is 
\begin{equation}\label{limiting_variance_1d}
\sigma^2(P,Q)=\int_0^1   \bar{c}^2_p(t;F,G)  dt.
\end{equation}
For a proof of the fact that $\sigma^2(P,Q)>0$ if and only if $P\ne Q$ we refer also to \cite{CLTonedimensional}. 

It remains only to prove \eqref{unbiased_CLT_1d_p1}. By Lemma \ref{Efron_Stein_L1} below $\sqrt{n}(\mathcal{T}_1(P_n,Q)-\mathbb{E}[\mathcal{T}_1(P_n,Q)])$ is stochastically bounded and from the comments after the proof we see that 
$${\rm Var}[Z_{n,M,1}]\leq \frac {C(P,M)}n,$$ where $Z_{n,M,1}=\int_{[-M,M]^C}|F_n(x)-G(x)|dx$ and $C(P,M)$ tends to zero as $M\to \infty$. We set also $$Z_{n,M,2}=\int_{-M}^M|F_n(x)-G(x)|dx.$$ 
Now, $$Z_{n,M,2}=\int_{-M}^M \left|F(x)-G(x)+\frac{\alpha_n(x)}{\sqrt{n}}\right|dx,$$
where $\alpha_n(x)=\sqrt{n}(F_n(x)-F(x))$ is the empirical process indexed by intervals $(-\infty,x]$, {with} $-M\leq x\leq M$. {We can see $\alpha_n=\frac 1 {\sqrt{n}}\sum_{i=1}^n (\mathbb{I}_{(X_i\leq x)}-F(x))$ as a random element with values in $L_1([-M,M])$, which is a cotype 2 Banach space. This implies that  $\mathbb{I}_{(X_i\leq \cdot)}-F(\cdot))$ is pregaussian if and only if
\begin{equation}\label{pregaussian}
\int_{-M}^{M} \sqrt{F(x)(1 - F(x))}\, dx<\infty,
\end{equation}
(see, e.g. p. 247 and p. 262  in \cite{ledoux1991probability}).
Condition \eqref{pregaussian} obviously holds. Hence (see Theorem 10.7 in \cite{ledoux1991probability}), since $L_1([-M,M])$ is of cotype 2 and $\mathbb{I}_{(X_i\leq \cdot)}-F(\cdot)$ is pregaussian,  $\alpha_n $ converges weakly to $ B \circ F$ in \( L_1([-M,M]) \), where \( B \) denotes a Brownian bridge on \( [0,1] \).
By Skorohod's theorem we can choose versions of $\alpha_n$ and $B\circ F$ (for which we will keep the same notation) such that $\int_{-M}^M |\alpha_n(x)-B\circ F(x)|\to 0$ a.s. as $n\to \infty$.} But then, if $$\tilde{Z}_{n,M,2}=\int_{-M}^M \left|F(x)-G(x)+\frac{B\circ F (x)}{\sqrt{n}}\right|dx$$
and we have that
\begin{equation}
     \label{equivalence_znM2}
\sqrt{n}\left|{Z}_{n,M,2}-\tilde{Z}_{n,M,2}\right|
\leq \int_{-M}^M 
\left|\alpha_n(x)-B\circ F(x)\right| dx\underset{a.s.}{\to} 0 .
\end{equation}
We observe next that 
\begin{equation}
    \label{limit_znM2}
{\sqrt{n}\left(\tilde{Z}_{n,M,2}-\int_{-M}^M \big|F(x)-G(x)\big|dx\right)}\underset{a.s.}{\to}  \int_{-M}^M v_{F,G}(x)dx,
\end{equation}
 with $v_{F,G}$ as in \eqref{def_v_FG}. In fact, 
\begin{eqnarray*}
\lefteqn{\sqrt{n}\Big(\tilde{Z}_{n,M,2}-\int_{-M}^M \big|F(x)-G(x)\big|dx\Big)}\hspace*{0.1cm}\\
&=& \int_{-M}^M \sqrt{n} \Bigg(\left|F(x)-G(x)+ \frac{B\circ F (x)}{\sqrt{n}}\right| -\big|F(x)-G(x)\big|\Bigg)dx.
\end{eqnarray*}
The integrand in the last expression converges pointwise to $\text{sgn}(F(x)-G(x))B(F(x))$ if $F(x)\ne G(x)$ and to $|B(F(x))|$ otherwise. Furthermore, it is upper bounded by $|B(F(x))|$ and the claim follows by dominated convergence.

In \eqref{equivalence_znM2} and \eqref{limit_znM2} we can easily check that we have also convergence of moments of order $r<2$. Hence, combining \eqref{equivalence_znM2}, \eqref{limit_znM2} and moment convergence (of all orders) we see that 
\begin{equation}\label{central_convergence}
\sqrt{n}(Z_{n,M,2}-\mathbb{E}[Z_{n,M,2}])\underset w \to \int_{-M}^M (v_{F,G}(x)-\mathbb{E}v_{F,G}(x))dx.
\end{equation}
Moment convergence ensures that the Wasserstein distance of order $r<2$ between the law of $\sqrt{n}(Z_{n,M,2}-\mathbb{E}[Z_{n,M,2}])$ and that of the random variable on the right tends to 0 as $n\to\infty$. 

Next, we set $$v_{F,G}^{(1)}(x)=|B\circ F(x)|\mathbb{I}(F(x)=G(x))$$ and $$v_{F,G}^{(2)}(x)=\text{sgn}(F(x)-G(x)) B\circ F(x) \mathbb{I}(F(x)\ne G(x)).$$ We  observe that for $M_1<M_2$
\begin{eqnarray*}
\lefteqn{V(M_1,M_2):=\mathbb{E}\left[\left(\int_{M_1}^{M_2}  (v_{F,G}(x)-\mathbb{E}[v_{F,G}(x)])dx\right)^2 \right]}\hspace*{1.5cm}\\
&\leq &2\mathbb{E}\left[\left(\int_{M_1}^{M_2}  (v^{(1)}_{F,G}(x)-\mathbb{E}[v^{(1)}_{F,G}(x)])dx\right)^2\right]\\&&+2 \mathbb{E}\left[\left(\int_{M_1}^{M_2}  v^{(2)}_{F,G}(x)dx\right)^2 \right]\\
&:=& 2 V_{1}+2V_{2}.
\end{eqnarray*}
The integral defining $V_{2}$ is a centered, Gaussian random variable and we can easily see that
$$V_{2}\leq \int_{M_1}^{M_2}\int_{M_1}^{M_2} (F(x\wedge y)-F(x)F(y))dxdy.$$
The integrand in the last expression is integrable over $\mathbb{R}^2$ (and the integral equals  the variance $\sigma^2(P)$). This entails that $V_2\to 0$ as $M_1,M_2\to \infty$. Similarly, using the fact that
$${\rm Cov}(|B(s)|,|B(t)|)\leq s\wedge t-st$$
for $0\leq s,t\leq 1$, { which follows easily from equation (6) in \cite{WellnerSmythe}}, we obtain that
$$V_{1}\leq \int_{M_1}^{M_2}\int_{M_1}^{M_2} (F(x\wedge y)-F(x)F(y))dxdy$$ 
and conclude that $V(M_1,M_2)\to 0$ as $M_1,M_2\to \infty$.
Now we can argue as in the proof of Theorem 5.1 in \cite{delBarrioGineMatran1999} to see that $\int_{-M}^M (v_{F,G}(x)-\mathbb{E}[v_{F,G}(x)])dx$ converges in $L_2$ as $M\to \infty$ to a random variable that we denote as
\begin{equation*}
\gamma(P,Q):=\int_{\mathbb{R}} (v_{F,G}(x)-\mathbb{E}[v_{F,G}(x)])dx.
\end{equation*}
Combining all the above estimates and using a classical $3\varepsilon$ argument we conclude that
$$\sqrt{n}(Z_{n}-\mathbb{E}[Z_{n}])\underset w \to G(P,Q).$$
\end{proof}

\begin{lemma}\label{Efron_Stein_L1} If $P$ and $Q$ are Borel probability measures on $\mathbb{R}^d$, $Q$ has finite mean and $P$ has finite variance $\sigma^2(P)$ and $P_n$ denotes the empirical measure on an i.i.d.~sample, $X_1,\ldots,X_n$, of observations from $P$, then
\begin{equation}\label{l1_var_bound}
\text{\em Var}(\mathcal{T}_1(P_n,Q))\leq \frac {\sigma^2(P)}n.
\end{equation}
\end{lemma}
\begin{proof} By duality 
$$Z:=\mathcal{T}_1(P_n,Q)=\sup_{f:\  , \|f\|_{\text{\small Lip}}\leq 1} \Big|\int f d(P_n-Q)\Big|.$$ 
By the Efron-Stein inequality $${\rm Var}[\mathcal{T}_1(P_n,Q)]\leq \frac n 2 \E\left[(Z-Z')^2_+ \right]$$ if $Z':=\mathcal{T}_1(P_n',Q)$ and $P_n'$ denotes the empirical d.f.~on $X_1',X_2,\ldots,$ $X_n$, with $X_1'$ a further independent observation from $P$. We observe that
$$|Z-Z'|\leq \sup_{f:\  , \|f\|_{\text{\small Lip}}\leq 1} \Big|\int f d(P_n-P_n')\Big|\leq {\textstyle\frac 1 n} \|X_1-X_1'\|,$$
which implies that 
$${\rm Var}[\mathcal{T}_1(P_n,Q)]\leq \frac{\mathbb{E}\left[\|X_1-X_1'\|^2 \right]}{2n}=\frac{\sigma^2(P)}n.$$
\end{proof}
In the case of probabilities on the real line we can reach the same conclusion writing $F_n$ and  $G$ for the d.f.'s associated to $P_n$ and $Q$, respectively, and recalling that $\mathcal{T}_1(P_n,Q)=\int_{\mathbb{R}}|F_n(x)-G(x)|dx=:Z$. With this representation we see that   
$$|Z-Z'|\leq \int_{\mathbb{R}} |F_n(x)-\tilde{F}_n(x)|dx=\frac{1}n |X_1-X_1'|'$$
and the Efron-Stein inequality yields again \eqref{l1_var_bound}. We can argue similarly to get a tighter control of the variance of
$$\tilde{\mathcal{T}}^{(M)}_1(P_n,Q):=\int_M^{\infty}|F_n(x)-G(x)|dx.$$
In fact, denoting now $Z=\tilde{\mathcal{T}}^{(M)}_1(P_n,Q)$ and $Z'$ as above we see that
\begin{eqnarray*}
|Z-Z'|&\leq &\frac 1 n \int_M^{\infty} |\mathbb{I}(X_1\leq x)-\mathbb{I}(X_1'\leq x)|dx\\
&\leq &\frac 1 n |X_1-X_1'|\mathbb{I}((X_1>M)\cup(X_1'>M)).
\end{eqnarray*}
Hence, we conclude that 
\begin{eqnarray*}
\lefteqn{n{\rm Var}(\tilde{\mathcal{T}}^{(M)}_1(P_n,Q))}\hspace*{1cm}\\
&\leq &\mathbb{E}\big(|X_1-X_1'|^2\mathbb{I}((X_1>M)\cup(X_1'>M))\big).
\end{eqnarray*}
We observe that the last upper bound vanishes as $M\to\infty$ if $P$ has a finite second moment.

\begin{proof}[Proof of Theorem \ref{teor:bias_1d}] The case $p>1$ is considered in \cite{sliced_arxiv}. For the case $p=1$ we argue as in \eqref{limit_znM2}, but now the assumption $J_1(P)<\infty$ allows to use weak convergence of the empirical process $\sqrt{n} (F_n-F)$ in $L_1(\mathbb{R})$ to conclude
$$\sqrt{n}(\mathcal{T}_1(P_n,Q)-\mathcal{T}_1(P,Q))\underset{w}\to \int_{\mathbb{R}}v_{F,G}(x)dx, $$ 
with convergence of moments of order $r$ for $r<2$. As a consequence we see that
$$  \sqrt{n}(\mathbb{E}[\mathcal{T}_1(P_n,Q)]-\mathcal{T}_1(P,Q))
    \longrightarrow \int_{F(x)=G(x)} \mathbb{E}|B(F(x))| dx.$$
\end{proof}

\begin{proof}[Proof of Theorems \ref{1dCLT_nullcase} and \ref{CLT_quantile_process}] To prove weak convergence of the process $v_n$ to $\mathbb{V}
$ as random elements in $L_p(0,1)$ we can argue as in the proof of   \cite[Theorem~2.5]{SamworthJohnson2004}, based, in turn, on \cite[Theorem~2.1]{CsorgoHorvath1993} to conclude that there are versions of $v_n$ and Brownian bridges such that 
\begin{equation}\label{strong_approximation_1}
\int_{1/n}^{1-1/n} \left|v_n(t)-\frac{B_n(t)}{f(F^{-1}(t))}\right|^pdt\underset {\text{\small Pr.}} \to 0.
\end{equation} 
The assumption $J_p(P){ <\infty}$ is sufficient (and necessary) to ensure that $\frac{B_n}{f\circ F^{-1}}$ is an $L_p$-valued random element and it is also clear from it that
\begin{equation}\label{tails_1}
\int_{[{1/n},{1-1/n}]^c} \left|\frac{B_n(t)}{f(F^{-1}(t))}\right|^pdt\underset {\text{\small Pr.}} \to 0 .  \end{equation}
Hence, to conclude, it only remains to show that 
$$\int_{[{1/n},{1-1/n}]^c} |v_n(t)|^pdt\underset {\text{\small Pr.}} \to 0.$$
But this is equivalent to showing that
\begin{equation}\label{tails_2}
n^{\frac{p}{2}}\int_{1-1/n}^1 |X_{(n)}-F^{-1}(t)|^p dt\underset {\text{\small Pr.}} \to 0,
\end{equation}
where $X_{(n)}$ denotes the maximum of the sample, with a similar condition for the lower tail. To check \eqref{tails_2} we observe that
\begin{eqnarray*}
\int_{1-1/n}^1 |X_{(n)}-F^{-1}(t)|^p dt&\lesssim & 
\int_{1-1/n}^1 |X_{(n)}-m_n|^p dt\\
&+&\int_{1-1/n}^1 |F^{-1}(t)-m_n|^p dt,
\end{eqnarray*}
with $m_n$ denoting the median of $X_{(n)}$. This shows that \eqref{tails_2} will follow if we prove that
\begin{equation}\label{tails_3}
n^{\frac{p-2}{2}}|X_{(n)}-m_n|^p \underset {\text{\small Pr.}} \to 0
\end{equation}
and
\begin{equation}\label{tails_4}
n^{ \frac{p}{2}}\int_{1-1/n}^1 |F^{-1}(t)-m_n|^p dt {\to 0}.
\end{equation}

To prove \eqref{tails_3} we observe that $X_{(n)}\overset d = F^{-1}(U_{(n)})$, with $U_{(n)}$ following a beta distribution with parameters $n$ and $1$ and use the variance bound in the proof of Proposition B.8 in \cite{BobkovLedoux2019} (the version centered at medians), which, in our setup means that
\begin{equation}\label{variance_bound_BL}
n^{\frac{p-2}{2}}\mathbb{E}|X_{(n)}-m_n|^p\lesssim \int_0^1 \frac{(t(1-t))^{\frac{p}{2}}}{f^p(F^{-1}(t))} t^{  n-1} dt.
\end{equation}
By dominated convergence the last integral vanishes { as $n\to\infty$}, proving \eqref{tails_3}. 

To prove \eqref{tails_4} we will use in the following inequality,
\begin{equation}\label{hardy_ineq}
\int_t^\infty (x-t)^p f(x)dx\leq p^p \int_{t}^\infty \Big(\frac{1-F(x)}{f(x)}\Big)^p f(x)dx,    
\end{equation}
which follows once we use integration by parts to see that
$$\int_t^\infty (x-t)^p f(x)dx=p\int_t^\infty (x-t)^{p-1}\frac{1-F(x)}{f(x)}f(x)dx$$
and {apply} then Hölder's inequality (see \cite{SamworthJohnson2004} for the case $p=2$). Using \eqref{hardy_ineq} (with the density of $X_{(n)}$, namely, $n F(x)^{n-1} f(x)$), we obtain that
\begin{eqnarray*}
\lefteqn{\mathbb{E}\left[\left|X_{(n)}-F^{-1}(1-{\textstyle \frac 1 n})\right|^p \mathbb{I}(X_{(n)}\geq F^{-1}({ 1-}{\textstyle \frac  1n})) \right]}\hspace*{0.1cm} \\
&{=} & { \int_{F^{-1}(1-\frac 1 n)}^\infty \big(x-\textstyle F^{-1}\big(1-\frac 1 n \big)\big)^pn F(x)^{n-1} f(x)dx}
\\
&{\leq} & { p^p\int_{F^{-1}(1-\frac 1 n)}^\infty \Big(\frac{1-F(x)^n}{n F(x)^{n-1} f(x)}\Big)^pnF(x)^{n-1}  f(x)dx}
\\
&{\lesssim} & {\int_{F^{-1}(1-\frac 1 n)}^\infty \Big(\frac{1-F(x)}{ f(x)}\Big)^pn F(x)^{n-1} f(x)dx}
\\
&{=} & { \int_{1-\frac 1 n}^1 \Big(\frac{1-t}{ f(F^{-1}(t))}\Big)^pn t^{n-1} dt}
\\
&{\lesssim} & { {\textstyle \frac {1} {n^{\frac{p-2}{2}}}}\int_{1-\frac 1 n}^1 \frac{(t(1-t))^{\frac p 2}}{ f^p(F^{-1}(t))} t^{n-1} dt}
\\
&{\lesssim}  & { {\textstyle \frac {1} {n^{\frac{p-2}{2}}}} \int_{1-\frac 1 n}^1 \frac{(t(1-t))^{\frac{p}{2}}}{f^p(F^{-1}(t))}dt.}
\end{eqnarray*}{
In the step between the third and the fourth lines we have used the fact that $1-t^n\leq n (1-t)$, $0\leq t\leq 1$ and also that, for $t\geq 1-\frac 1 n$,  $t^{n-1}\geq (1-\frac 1 n)^{n-1}\to \frac 1 e>0$; later, the step between the fifth and the sixth lines follows upon observing that $(1-t)^{\frac p 2}\leq n^{-\frac p 2}$ if $1-\frac 1 n \leq t\leq 1$. Now, since $J_p(P)$ is finite we conclude that 
\begin{eqnarray}
\nonumber
\lefteqn{n^{\frac{p-2}{2}}\mathbb{E}\left[\left|X_{(n)}-F^{-1}(1-{\textstyle \frac 1 n})\right|^p \mathbb{I}(X_{(n)}\geq F^{-1}({ 1-}{\textstyle \frac  1n})) \right]}\hspace*{4.1cm} \\ \label{extremes_1}
&{\lesssim}  & { \int_{1-\frac 1 n}^1 \frac{(t(1-t))^{\frac{p}{2}}}{f^p(F^{-1}(t))}dt\to 0.}
\end{eqnarray}
}
Observe that \eqref{variance_bound_BL} shows also that
\begin{equation}\label{extremes_2}
n^{\frac {p-2} 2}\mathbb{E}\left[\left|X_{(n)}-m_n\right|^p \mathbb{I}(X_{(n)}\geq F^{-1}({1-\textstyle \frac  1n})) \right]\to 0. 
\end{equation}
But now we can see that
\begin{align*}
  &  n^{\frac{p-2}{2 p}}\left(\mathbb{P}\left(X_{(n)}\geq F^{-1}\left({ {\textstyle 1- \frac  1n}}\right)\right)\right)^{\frac{1}{p}}\\
 & \qquad \qquad  \qquad \qquad \times \left|m_n-F^{-1}\left(1-{\textstyle \frac 1 n }\right)\right|\\
&= n^{\frac{p-2}{2 p}} \Big(\mathbb{E} \Big[\left|m_n-F^{-1}\left(1-{\textstyle \frac 1 n }\right) \right|^p \\
& \qquad \qquad  \qquad \qquad \times\mathbb{I}\left(X_{(n)}\geq F^{-1}({1-\textstyle \frac  1n})\right) \Big]\Big)^{\frac{1}{p}}\\
&\leq \left( n^{\frac{p-2}{2}} \mathbb{E}\left[\left|m_n-X_{(n)}\right|^p \mathbb{I}(X_{(n)}\geq F^{-1}\left({1-\textstyle \frac  1n}\right) \right] \right)^{\frac{1}{p}}\\
&+
\Big(n^{\frac{p-2}{2}} \mathbb{E} \Big[\left|X_{(n)}-F^{-1}\left(1-{\textstyle \frac 1 n }\right)\right|^p\\
&\qquad  \qquad \qquad \qquad \times \mathbb{I}\left(X_{(n)}\geq F^{-1}\left({\textstyle  1-\frac  1n}\right) \Big) \right]  \Big)^{\frac{1}{p}}
\end{align*}
and we conclude, using again that $\mathbb{P}\left(X_{(n)}\geq F^{-1}({ {\textstyle 1- \frac  1n}})\right)$ $\to \frac 1 e>0$,
\begin{equation}\label{extremes_3}
n^{\frac{p-2}{2}}\left|m_n-F^{-1}\left(1-\textstyle\frac 1 n \right)\right|^p \to 0.
\end{equation}
This implies that we can replace $m_n$ with $F^{-1}(1-\textstyle \frac 1 n)$ for proving \eqref{tails_4}. But then, from \eqref{hardy_ineq} we see that
\begin{eqnarray*}
\lefteqn{n^{\frac{p}{2}}\int_{1-\frac 1 n}^1 |F^{-1}(t)-F^{-1}(1-{\textstyle \frac 1 n})|^pdt }\hspace*{2cm}\\
&\lesssim & n^{p/2 }\int_{1-1/n}^1 \frac{(1-t)^p}{f^p(F^{-1}(t))}dt\\
&\leq & \int_{1-\frac 1 n}^1 \frac{(1-t)^{\frac{p}{2}}}{f^p(F^{-1}(t))}dt\to 0.
\end{eqnarray*}
This completes the proof of Theorem \ref{CLT_quantile_process}. Theorem \ref{1dCLT_nullcase} follows by the continuous mapping theorem.
\end{proof}

\putbib[supprefs]

\end{bibunit}

\end{document}